\documentclass[11pt, reqno]{amsart}
\usepackage{latexsym}

\usepackage{amsmath}
\usepackage{color}

\definecolor{blu}{rgb}{0,0,0.1}

\makeatletter
\def\itemize{
  \ifnum\@itemdepth>3\@toodeep\else
    \advance\@itemdepth\@ne
    \edef\@itemitem{labelitem\romannumeral\the\@itemdepth}%
        \list{\csname\@itemitem\endcsname}%
      {\leftmargin=20pt\def\makelabel##1{\hss\llap{##1}}}
        \fi}

\renewenvironment{enumerate}{%
  \ifnum \@enumdepth >3 \@toodeep\else
      \advance\@enumdepth \@ne
      \edef\@enumctr{enum\romannumeral\the\@enumdepth}\list
      {\csname label\@enumctr\endcsname}{\usecounter
        {\@enumctr}\leftmargin=20pt\def\makelabel##1{\hss\llap{\upshape##1}}}\fi
}{%
  \endlist
}

\makeatletter
\def\@settitle{\begin{center}%
  \baselineskip14\p@\relax
    \bfseries\@title
  \end{center}%
}
\def\@setauthors{%
  \begingroup
  \def\thanks{\protect\thanks@warning}%
  \trivlist
  \centering\footnotesize \@topsep30\p@\relax
  \advance\@topsep by -\baselineskip
  \item\relax
  \author@andify\authors
  \def\\{\protect\linebreak}%
  \authors%
  \ifx\@empty\contribs
  \else
    ,\penalty-3 \space \@setcontribs
    \@closetoccontribs
  \fi
  \endtrivlist
  \endgroup
}
\def\@setthanks{\def\thanks##1{\par##1}\thankses}
\makeatother

   \makeatletter
   \def\LaTeX{\leavevmode L\raise.42ex
       \hbox{\kern-.3em\size{\sf@size}{0pt}\selectfont A}\kern-.15em\TeX}
   \makeatother
   
   \newcommand{\BibTeX}{{\rm B\kern-.05em{\sc
             i\kern-.025emb}\kern-.08em\TeX}}
\def\bbm[#1]{\mbox{\boldmath $#1$}}



   \newcommand{\R}{{\mathbb{R}}}

    \newcommand{\sphe}{\mathbb{S}}
\newcommand{\cal}{\mathcal }  
   
   \newcommand{\N}{\mathbb{N}}
   \newcommand{\beq}{\begin{equation}}
   \newcommand{\eeq}{\end{equation}}
 
 \newcommand{\Ree}{\mathrm{Re}}

 \renewcommand{\theequation}{\arabic{section}.\arabic{equation}}
 \renewcommand{\theequation}{\thesection.\arabic{equation}}
   \newtheorem{theorem}{Theorem}[section]
   
   \newtheorem{proposition}[theorem]{Proposition}
   \newtheorem{lemma}[theorem]{Lemma}
      \newtheorem{example}[theorem]{Example}
   
   \newtheorem{remark}[theorem]{Remark}
   \newcommand{\bremark}{\begin{remark} \em}
   \newcommand{\eremark}{\end{remark} }

\def\bbm[#1]{\mbox{\boldmath $#1$}}
\def\bbm[#1]{\mbox{\boldmath $#1$}}

\begin{document}

\title[Equilibria of point-vortices on closed surfaces]
{\LARGE Equilibria of point-vortices on closed surfaces}

\author[T. D'Aprile \& P. Esposito]{\Large Teresa D'Aprile \and Pierpaolo Esposito}
\address{Teresa D'Aprile, Dipartimento di Matematica, Universit\`a di Roma ``Tor
Vergata", via della Ricerca Scientifica 1, 00133 Roma, Italy.}
\email{daprile@mat.uniroma2.it}

\address{Pierpaolo Esposito, Dipartimento di Matematica e Fisica, Universit\`a degli Studi ``Roma Tre", Largo S. Leonardo Murialdo 1, 00146 Roma, Italy.}
\email{esposito@mat.uniroma3.it }
\begin{abstract} We discuss the existence of equilibrium configurations for the Hamiltonian point-vortex model on a closed surface $\Sigma$. The topological properties of $\Sigma$ determine the occurrence of three distinct situations, corresponding to $\mathbb{S}^2$, to $\mathbb{RP}^2$ and to $\Sigma \not=\mathbb{S}^2,\mathbb{RP}^2$. As a by-product, we also obtain new existence results for the singular mean-field equation with exponential nonlinearity. 

\medskip



\end{abstract}
\maketitle

\section{introduction} Let $\Sigma$ be a closed surface (i.e. compact and without boundary) endowed with a metric tensor $g$. We are concerned with equilibrium configurations of the Hamilton function 
$$
\mathcal{H}_0(\bbm[\xi])= \sum_i \Gamma_i^2 H(\xi_i,\xi_i)+\sum_{i \neq j} \Gamma_i \Gamma_j G(\xi_i,\xi_j)
$$
for $\bbm[\xi]=(\xi_1,\dots,\xi_{N_0}) \in \Sigma^{N_0} \cap \{ \xi_i \not= \xi_j \hbox{ for }i\not=j \}$, where $G(x,p)$ is the Green's function of $-\Delta_g$ over $\Sigma$ with singularity at $p$ and $H(x,p)$ is its regular part.

\medskip In an inviscid and incompressible fluid, the velocity field and the pressure obey the Euler equations. For a two-dimensional turbulent flow, the point-vortex ansatz $\omega=\sum_{i=1}^{N_0} \Gamma_i \delta_{\xi_i(t)}$ for the (scalar) vorticity function $\omega$ leads to the following Hamiltonian system:
\begin{equation} \label{Hamilton}
\Gamma_i \, \partial_t  \xi_i=J \, \nabla_{\xi_i} \mathcal{H}_0(\bbm[\xi]) \qquad \forall\, i=1,\dots,N_0,
\end{equation}
where $J$ denotes the symplectic matrix
$$J=\left(\begin{array}{ll} 0&1 \\ -1 & 0 \end{array} \right).$$
The quantity $\Gamma_i \in \mathbb{R} \setminus \{0 \}$ is the strength of the point-vortex $\xi_i$, whose sign determines the clockwise/counterclockwise rotation of the fluid near $\xi_i$. Based on ideas of Helmoltz (\cite{Hel}), \eqref{Hamilton} has been derived by Kirchhoff (\cite{Kir}) in $\mathbb{R}^2$. Extended by Routh (\cite{Rou}) to a bounded domain in terms of the so-called hydrodynamic Green function (see also \cite{Lin1,Lin2}), the renormalized kinetic energy $\mathcal{H}_0$ is referred to as the Kirchhoff-Routh path function.  The interested reader can look at \cite{ANSTV,New} for the case of a surface (like spheres, cylinders or tori), and refer to \cite{FlGu,MaBe,MaPu,New,Saf} for a modern treatment of the topic.

\medskip Apart from $\mathbb{R}^2$ and the case of special domains (like discs, half-discs, annuli, strips), very few is known concerning the existence of equilibrium configurations for $\mathcal{H}_0$. On a closed surface, notice that $\mathcal{H}_0$ has always a minimum point when the point-vortices have the same orientation (say, $\Gamma_i \geq 0$ for all $i=1,\dots, N_0$). The presence of counter rotating vortices makes the problem very difficult. On a bounded domain, when $N_0\leq 4$ point-vortices of alternating orientations have been considered in \cite{BPW} with $\Gamma_i=(-1)^i$ and in \cite{BaPi} for the general case (see also \cite{BrPi} for $N_0=2$). The assumption on $N_0$ prevents the collision of some $\xi_i$'s with opposite orientations, the simplest case being given by three point-vortices with $\Gamma_i=1$ collapsing onto one with $\Gamma_i=-1$ (see \cite{EsWe} in a PDE context).

\medskip In this paper we address the case where all the point-vortices with negative orientation are kept fixed. Denoting them by $p_1,\dots,p_\ell$ with strengths $-\frac{\alpha_1}{2}, \dots, -\frac{\alpha_\ell}{2}$, we are led to study
\beq\label{psi0}
\mathcal{H}(\bbm[\xi])= \sum_{j,k=1 \atop j \neq k}^N \Gamma_j \Gamma_k G(\xi_j,\xi_k)-\sum_{i=1}^\ell \alpha_i \sum_{j=1}^N \Gamma_j G(\xi_j,p_i)+\sum_{j=1}^N h(\xi_j)
\eeq
for $\bbm[\xi]=(\xi_1,\dots,\xi_N) \in \mathcal M$, where $N=N_0-\ell$, $\alpha_i,\, \Gamma_i>0$, $h \in C^1(\Sigma,\mathbb{R})$ and
$$\mathcal M=(\Sigma \setminus \{p_1,\dots,p_\ell\})^N \setminus \Delta, \quad \Delta=\{ \bbm[\xi] \in \Sigma^N:\ \xi_j=\xi_k \hbox{ for some }j\not= k\}.$$
Inspired by some arguments in \cite{BDM,BJMR}, the main aim of our paper is to investigate the interaction of the topology of $\Sigma$ with the presence of singular sources $p_1,\dots,p_\ell$ toward the existence of equilibria for $\mathcal H$. As we will see below, the three cases $\Sigma =\mathbb{S}^2$, $\Sigma=\mathbb{RP}^2$ and $\Sigma \not=\mathbb{S}^2,\mathbb{RP}^2$ exhibit completely different phenomena.

\medskip The critical point of $\mathcal{H}$ will be found at the max-min energy level  
$${ \mathcal H}^*=\sup_{\gamma\in{\mathcal F}}\min_{\hbox{\scriptsize$\bbm[\xi]$}\in
K}\mathcal H(\gamma(\bbm[\xi])),$$
where $\mathcal F$ collects a suitable family of  deformation maps from $K$ into an open set $\mathcal D \subset \subset \mathcal M$ that keep fixed $K_0 \subset K$ (for some compact sets $K$, $K_0$).
To prevent the collapsing for part of the $\xi_j$'s onto some $p_i$, the following compactness condition is crucial:
\begin{equation} \label{compact}
\alpha_i \notin \Bigg\{ \Bigg(\sum_{j,k \in J \atop j\not=k} \Gamma_j \Gamma_k \Bigg) \bigg(\sum_{j \in J} \Gamma_j \bigg)^{-1} :\ J \subset \{1,\dots,N\} \Bigg\} \qquad \forall \,  i=1,\dots,\ell.
\end{equation}
When $\Gamma_1=\dots=\Gamma_N=1$, notice that \eqref{compact} simply reduces to 
\beq\label{akka1}\alpha_i \neq 1,\dots, N-1 \quad \forall\: i=1,\dots, \ell.\eeq
To produce the linking structure
$${ \mathcal H}^*<\min_{\hbox{\scriptsize$\bbm[\xi]$}\in K_0} \mathcal H(\bbm[\xi]),$$ we need that a crucial intersection property is accomplished: more precisely, by applying a topological degree argument, 
 for all $\gamma \in \mathcal F$ we catch  a point $\bbm[\xi]_\gamma^* \in K$ with prescribed $\mathcal P_j(\gamma_j(\bbm[\xi]_\gamma^*))$, $j=1,\dots,N$, for suitable retraction  maps $\mathcal P_j$. When $\Sigma \not= \mathbb{S}^2, \mathbb{RP}^2$, we take $\mathcal P_j=\mathcal P$ for all $j=1,\dots,N$, $\mathcal P$ being a retraction of $\Sigma$ onto a simple closed curve $\sigma \subset \Sigma \setminus \{p_1,\dots,p_\ell\} $. Since the fibers of $\mathcal P$ are well separated, the value ${ \mathcal H}^*$ is uniformly (with respect to $K_0$) bounded from above, whereas $\min_{\hbox{\scriptsize$\bbm[\xi]$}\in K_0} \Psi(\bbm[\xi])$ can be made arbitrarily large by a suitable choice of $K_0$. Our first main result then reads as follows:
\begin{theorem}\label{mainth1} Let $\Sigma$ be a closed surface topologically different from $\mathbb{S}^2$ and $\mathbb{RP}^2$. If \eqref{compact} does hold, then $\mathcal H$ has a critical point.
\end{theorem}
\noindent 
When $\Sigma=\mathbb{RP}^2$, every map $\mathcal P_j$, $j=1,\dots, N$, can be taken instead as a retraction $\mathcal P$ of $\mathbb{RP}^2 \setminus \{p_i \}$ onto a simple closed curve $\sigma \subset \Sigma \setminus \{p_1,\dots,p_\ell\}$ for a fixed $i =1,\dots,\ell$. 
In this case the fibers of $\mathcal P$ are curves emanating from the singular source  $p_i$ and the assumption 
\begin{equation} \label{1530b} \sum_{j,k =1 \atop j \not= k}^N \Gamma_j \Gamma_k<\alpha_i \sum_{j =1}^N\Gamma_j 
\end{equation}
is required to assure that the mutual interactions between the components of $\bbm[\xi]$ are dominated by  the interplay between each component with  $p_i$, 
 which is essential  to get a uniform control from above on $\mathcal H^*$. 
So, our second main result is the following:
\begin{theorem}\label{mainth2} Let $\Sigma$ be a closed surface topologically equivalent to $\mathbb{RP}^2$. If \eqref{compact} and 
\begin{equation} \label{1530} 
\Bigg(\sum_{j,k =1 \atop j \not= k}^N \Gamma_j \Gamma_k \Bigg) \bigg( \sum_{j =1}^N \Gamma_j \bigg)^{-1} <\max\{\alpha_1,\dots, \alpha_\ell \}
\end{equation}
do hold, then $\mathcal H$ has a critical point.
\end{theorem}
\noindent The Euclidean case (\cite{Dap,DKM}), which has been the starting point for our investigation, has a strong analogy with $\mathbb{RP}^2$. When $\Gamma_1=\dots=\Gamma_N=1$,  \eqref{1530} becomes
$$ N-1<\max\{\alpha_1,\dots,\alpha_\ell \}.$$

The case $\Sigma=\mathbb{S}^2$ is more involved since $\mathbb{S}^2 \setminus \{ p \}$ is contractible, therefore it is essential that we remove two points from $\mathbb{S}^2$ in order to find a suitable retraction. 
So, first of all  we need to assume $\ell\geq 2$. Let us split $$\{1,\dots,N\}={\cal N}_1\cup\dots \cup{\cal N}_\ell$$ with disjoint union and set $N_i=\#{\cal N}_i\geq 0$. Then,  each $i=1,\dots,\ell$ has to be coupled with $r(i) \not=i$ and we choose  $\mathcal P_j={\cal P}_i$ for all $j\in {\cal N}_i$, $\mathcal P_i$ being a retraction of $\mathbb{S}^2 \setminus \{p_i,p_{r(i)} \}$ onto a simple closed curve $\sigma_i \subset \mathbb{S}^2 \setminus \{p_1,\dots,p_\ell\}$. The fibers of $\mathcal P_i$ are curves between $p_i$ and $p_{r(i)}$, and then part of such $N_i$ points could approach not only $p_i$ but also $p_{r(i)}$. Now, by exchanging the role of $i$ and $r$, for every $i=1,\dots,\ell$  we define the set $J_i \subset \{1,\dots,\ell\} \setminus \{i \}$,  as made up of those indices which are coupled in the above construction with $i$, and then $\{1,\dots,\ell\}$ is the disjoint union of such $J_i$, $i=1,\dots,\ell$. 
To obtain an upper bound on $\mathcal H^*$ we need to require
\begin{equation} \label{0956}
 \sum_{j,k \in \tilde{\mathcal N}_i \atop j \not= k} \Gamma_j \Gamma_k<\alpha_i \sum_{j \in \tilde{\mathcal N}_i} \Gamma_j \quad \;\forall i=1,\dots,\ell,
\end{equation}
where 
$$\tilde{\mathcal N}_i= \mathcal N_i \cup \bigcup_{r \in J_i} \mathcal N_r.$$
When $\Gamma_1=\dots=\Gamma_N=1$, notice that \eqref{0956} turns into
\beq\label{akka3} N_i+\sum_{r \in J_i} N_r-1<\alpha_i   \quad \forall\: i=1,\dots, \ell.\eeq
Our third main result reads as:
\begin{theorem}\label{mainth3} Let $\Sigma$ be a closed surface topologically equivalent to $\mathbb{S}^2$ and $\ell\geq 2$. If \eqref{compact} and \eqref{0956} do hold, then $\mathcal H$ has a critical point.
\end{theorem}

\medskip Hereafter, we restrict our attention to the case $\Gamma_1=\dots=\Gamma_N=1$. The corresponding $\mathcal H$ can also be seen as the reduced energy (\cite{cl1,cl2,espofigue}) for the  following singular mean-field problem
\beq\label{liouv}
-\Delta_g u=\lambda\bigg(\frac{\kappa(x)e^{u}}{\int_{\Sigma} \kappa(x) e^u\,dV_g}-\frac{1}{|\Sigma|}\bigg)-4\pi\sum_{i=1}^\ell\alpha_i\bigg(\delta_{p_i}-\frac{1}{|\Sigma|}\bigg)
\eeq
when looking for solutions blowing-up at distinct points $\xi_1,\dots,\xi_N \in \Sigma \setminus \{p_1,\dots,p_\ell\}$. Here, $\lambda$ is a parameter close to $8\pi N$, $\kappa:\Sigma\to \R$ is a smooth positive function, and $p_1,\dots ,p_\ell \in \Sigma$ are singular sources with $\alpha_i>0$. We denote by $\delta_p$  the Dirac measure supported at $p$, by $dV_g$  the area element in $(\Sigma,g)$ and by $|\Sigma|=\int_\Sigma dV_g$  the area of $\Sigma$. 

\medskip Regular mean-field equations naturally arise  in conformal geometry (\cite{CGY,CY,KW}), in statistical mechanics (\cite{CLMP1,CLMP2,ChKi,Kie}) and in the study of turbulent Euler flows (\cite{Stuart}). The singularities can model Euler flows interacting with sinks of opposite vorticities (\cite{TuYa}) or conical singularities on a surface (\cite{BDM,ChLi,Tro}), and naturally arise in connection with the Chern-Simons-Higgs model (\cite{CLW,DJLW,LiWa,NoTa,Tar,Yan}) and the Electroweak theory (\cite{BrDM,BrTa,SpYa}). To attack existence issues, one can compute the topological degree (\cite{cl1,cl2,cl3,cl4,CLW}), use a min-max variational approach (\cite{BDM,BrMa,BrTa,DJLW2,MaRu}) or perturbative arguments in the regime $\lambda \to 8\pi N$ (\cite{cl1,cl2,espofigue}, see also \cite{BaPa,DEFM,DEM,DKM,EGP,LiYa}). The topological degree is non-zero when $\Sigma \not= \mathbb{S}^2, \mathbb{RP}^2$, $\alpha_j \in \mathbb{N}$ and $\lambda \notin 8\pi \mathbb{N}$ (\cite{cl4}). For $\mathbb{S}^2$ this is still true (\cite{CLW}) when $\lambda \in (8\pi,16\pi)$ and $\ell\geq 2$, but the topological degree vanishes in several cases like:
\begin{itemize}
\item $\ell=1$ and $\lambda \in (8\pi, 8\pi(1+\alpha_1)) \cup (8\pi (2+\alpha_1),+\infty)$ (which is consistent with the necessary condition (\cite{Tar1}) for the existence:  $\lambda \leq 8\pi$ or $\lambda \geq 8\pi(1+\alpha_1)$);
\item $\ell=2$ and $\lambda \in (8\pi(1+\alpha_1),8\pi (1+\alpha_2)) \cup (8\pi(2+\alpha_1+\alpha_2),+\infty)$ if $\alpha_1\leq \alpha_2$ (in agreement with the necessary condition (\cite{BrMa}) for the existence: $\lambda<8\pi(1+\alpha_1)$ or $\lambda>8\pi(1+\alpha_2)$).
\end{itemize}
In a similar way, the topological degree can vanish when $\Sigma=\mathbb{RP}^2$. An alternative variational approach is also available, which is completely general  for $\Sigma\not= \mathbb{S}^2,\mathbb{RP}^2$ (\cite{BDM}, see also \cite{BJMR}) and requires the following restrictions (\cite{BrMa,MaRu}) when $\Sigma=\mathbb{S}^2$:
\begin{enumerate}
\item[a)] $\alpha_1,\dots,\alpha_\ell \leq 1$, $\lambda \in (8\pi,16\pi) \setminus \{8\pi(1+\alpha_1),\dots,8\pi(1+\alpha_\ell) \}$ and $\# J_\lambda\geq 2$, where
$$J_\lambda=\{i=1,\dots,\ell: \lambda<8\pi(1+\alpha_i)\};$$ 
\item[b)] $\ell\geq 2$ and $\lambda\in (0,8\pi \displaystyle \min_{i=1,\dots,\ell} (1+\alpha_i) )\setminus 8\pi \mathbb{N}$, which can be stated equivalently as $\# J_\lambda=\ell \geq 2$.
\end{enumerate}

In the special regime $\lambda \to 8\pi N$ solutions of \eqref{liouv} may possibly exhibit concentration phenomena, a property of definite physical interest since the right hand side of \eqref{liouv} represents precisely the vorticity of the Euler flow. The concentration points, if different from $p_1,\dots,p_\ell$, have to correspond to a critical point of a reduced energy having $\mathcal H$ (with $h(\xi)=H(\xi,\xi)+\frac{1}{4\pi}\log \kappa(\xi)$) as main order term. The existence of such concentrated solutions has been addressed, among other things, in \cite{cl2} for non-degenerate critical points $\bbm[\xi]=(\xi_1,\dots,\xi_N)$ of $\mathcal H$ with non-vanishing
$$\begin{aligned}
A(\bbm[\xi])
=\sum_{j=1}^N \kappa(\xi_j)e^{8\pi H(x,\xi_j)-4\pi \sum_{i=1}^\ell \alpha_i G(\xi_j,p_i)+8\pi\sum_{k\neq j}G(x,\xi_k)}\!\bigg[\Delta_g \log \kappa(\xi_j)+\!4\pi\frac{2 N\!- \!\Lambda}{|\Sigma|}\!-\!2K(\xi_j)\bigg] ,
\end{aligned}$$
where $\Lambda=\sum_{i=1}^\ell \alpha_i$ and $K$ is the Gaussian curvature of $(\Sigma, g)$. The critical points provided by Theorems \ref{mainth1}, \ref{mainth2} ,\ref{mainth3} may possibly be degenerate, but the critical value $\mathcal H^*$ is stable with respect to small $C^1$-perturbations of $\mathcal H$,  since it has been found by a max-min scheme. Thanks to the result in \cite{espofigue}, stable critical values of $\mathcal H$, under the sign assumption $A<0$  ($A>0$, respectively), give rise to a family of solutions $u_\lambda$ for \eqref{liouv} such that
$$\frac{\lambda \kappa(x)e^{u_\lambda}}{\int_{\Sigma}\kappa(x)e^{u_\lambda}dV_g}\to 8\pi \sum_{j=1}^N\delta_{\xi_j}$$
as $\lambda\to 8\pi N^-$ ($\lambda\to 8\pi N^+$, respectively)  in the measure sense, for a critical point $\bbm[\xi]$ of $\mathcal H$ with $\mathcal H(\bbm[\xi])=\mathcal H^*$. 
Consequently, as a byproduct of Theorems \ref{mainth1}, \ref{mainth2} and \ref{mainth3}, we provide solutions of multi-bubble type to equation \eqref{liouv} in the special regime $\lambda \to 8\pi N$. In many cases, we obtain the perturbative counterparts of global existence results already available in literature, obtained via degree theory or a global variational approach. However, compared to such previous results, in some situations one can still have existence when the degree of the equation vanishes even beyond the threshold on  $\lambda$ imposed by  ${\rm a)}$ and ${\rm b)}$, as we will see by explicit examples (see Remark \ref{rem} and Example \ref{exe} below). 

\medskip Setting
$$[\alpha]^-=\max\{n\,|\, n\in \mathbb{Z},\; n<\alpha\}\qquad \forall \alpha\in\R,$$
 for $\Gamma_1=\dots=\Gamma_N=1$ 
we summarize Theorems \ref{mainth1}, \ref{mainth2} and \ref{mainth3} as follows:
\begin{theorem}\label{thgammauno}
Assume that \eqref{akka1} holds for $N$. Then
$$\mathcal{H}(\bbm[\xi])= \sum_{j,k=1 \atop j \neq k}^N G(\xi_j,\xi_k)-\sum_{i=1}^\ell \alpha_i \sum_{j=1}^N G(\xi_j,p_i)+\sum_{j=1}^N h(\xi_j)$$
has a $C^1-$stable critical value $\mathcal H^*$ if 
\begin{itemize}
\item $\Sigma\neq \mathbb{S}^2, \mathbb{RP}^2$;
\item $\Sigma= \mathbb{RP}^2$ and
$$
N \leq \max\{1+[\alpha_1]^-,\dots, 1+[\alpha_\ell]^- \};
$$
\item $\Sigma= \mathbb{S}^2$, $\ell\geq 2$ and
\begin{equation}
N_i+\sum_{r \in J_i}N_r \leq 1+[\alpha_{i}]^- \qquad \forall \, i=1,\dots,\ell, \label{192}
\end{equation}
for $N_i\in \mathbb{N}\cup\{0\}$ and disjoint subsets $J_i\subset \{1,\dots, \ell\}\setminus\{i\}$, $i=1,\dots,\ell$, so that 
$$N=\sum_{i=1}^\ell N_i, \qquad \{1,\dots,\ell\}= \bigcup_{i=1}^\ell J_i.$$
\end{itemize}
\end{theorem}
\noindent 
\begin{remark}\label{rem} For $\lambda\to 8\pi N$, on the standard sphere $(\mathbb{S}^2, g_0)$ let us compare \eqref{192} with $\rm a)-\rm b)$.  Assumption \eqref{192} turns out  to be more general for $\ell\geq 3$  by allowing larger values of $\lambda$:  notice that $N$ may possibly overcome the value $\max_{i}(1+\alpha_i)$. Moreover, the choice $N_i=0$ for all $i \neq i_1, i_2$ and $J_{i_1}=\{i_2\}$, $J_{i_2}=\{i_1\}$ shows that $\#J_\lambda \geq 2$ implies the validity of \eqref{192}, extending $\rm a)$ to general $\alpha_i$ and $\rm b)$. Since in this case $A<0$ if $\kappa=1$ in view of $K=4\pi|\Sigma|^{-1}$, Theorem \ref{thgammauno} provides new existence results for equation \eqref{liouv} on $(\mathbb{S}^2,g_0)$ with $\kappa=1$ when $\lambda\to 8\pi N$  compared to \cite{BrMa,MaRu}.
\end{remark}
\begin{example}\label{exe} Consider equation \eqref{liouv} on the standard sphere $(\mathbb{S}^2, g_0)$ with $\kappa=1$ and let $$\ell=3,\quad \alpha_1=\alpha_2=\alpha\in (1,2),\quad \alpha_3 \geq 4.$$
According to the degree formula computed by Chen-Lin (\cite{cl4}), it can be easily checked that the degree vanishes for  $\lambda\in(8\pi (2+\alpha), 32\pi)$; moreover the existence results in \cite{BrMa,MaRu} do not work in such an interval since neither  assumption ${\rm a)}$ nor assumption ${\rm b)}$ are satisfied.  On the other hand, \eqref{192} is verified by taking $J_1=\emptyset$, $J_2=\{3\}$, $J_3=\{1,2\}$, and $N_1=N_2=2$, $N_3=0$, and it is immediate to check that $A<0$ if $\kappa=1$. Then, as a byproduct of Theorem \ref{thgammauno}, we deduce the existence of a solution to the Liouville equation \eqref{liouv} for $\lambda$ in a small left neighborhood of $32\pi$  with $N=4$ blow-up points. This example provides a new existence result in a perturbative regime  for equation \eqref{liouv} which is not covered neither by the degree theory  (\cite{cl4})  nor by variational methods (\cite{BrMa,MaRu}). 
\end{example}
Assumption \eqref{192} comes from \eqref{akka3} but is quite involved in such a general form. Finally, consider Theorem \ref{thgammauno}  restricted  to the case $\# J_{i}=1$, $i=1,\dots,\ell$, which, up to re-ordering, simply means that each $p_i$ is coupled (in the construction of $\mathcal P_i$) with $p_{i+1}$ (with the convention $p_{\ell+1}=p_1$, $\alpha_{\ell+1}=\alpha_1$ and $N_{\ell+1}=N_1$). Referred to as a consecutive coupling of the $p_i$'s, assumption \eqref{192} reduces to
\begin{equation}
N_i+N_{i+1}\leq 1+[\alpha_{i+1}]^-\qquad
 \forall\, i=1,\dots, \ell .\label{1922}
\end{equation}
From now on we will use the following notation: the quantities $a_i=1+[\alpha_i]^- $, $i=1,\dots,\ell$, correspond to a general consecutive coupling, whereas $b_1,\dots,b_\ell$ will denote the increasing ordering with $\alpha_1\leq\dots \leq \alpha_\ell$. Given $J \subset \{1,\dots,k\}$, for any $2\leq k\leq \frac{\ell}{2}$ let us define
$$s_k(J)=\sum_{j=2}^k \left[ \chi_J(j) a_{2j}+(1-\chi_J(j)) \left( a_{2j+1}+\chi_J(j-1) \min \{a_{2j-1},a_{2j}\} \right) \right],$$
where $\chi_J$ denotes the characteristic function of $J$. Set $c_1=a_1$, $g_1=a_3$, $d_1=f_1=+\infty$, and for $k\geq 2$
$$c_k=\min\{a_2+s_k(J): \: J \hbox{ s.t. } 1,k \in J\}, \quad d_k=\min\{a_3+s_k(J): \: J\hbox{ s.t. }1 \not\in J, k \in J \}$$
and
$$f_k=\min\{a_2+s_k(J): \: J\hbox{ s.t. }1 \in J,  k \not\in J\}, \quad g_k=\min\{a_3+s_k(J): \: J\hbox{ s.t. } 1 , k \notin J \}.$$
In order to determine the maximal $N=\sum_{i=1}^\ell N_i$ so that \eqref{1922} holds, in Appendix A simple but involved computations show the following:
\begin{theorem} \label{maximizingN}
For a general consecutive coupling, the maximal $N$ is given by
$$N=\min \{ c_{\frac{\ell}{2}},g_{\frac{\ell}{2}}\}$$
for $\ell$ even and
$$N=\max\Big\{ \min \{ c_{\frac{\ell-1}{2}}+\hat a_1-N_1,d_{\frac{\ell-1}{2}}+\hat a_1,f_{\frac{\ell-1}{2}},g_{\frac{\ell-1}{2}}+N_1 \}: \hat a_1 -\min\{a_1,a_\ell\}\leq  N_1 \leq  \min\{a_1,a_2\} \Big\}$$
for $\ell$ odd, where $\hat a_1=\min \{a_1,\min\{a_1,a_2\}+\min\{a_1,a_\ell\} \}$. For the increasing consecutive coupling, $N$ takes the form
\begin{equation} \label{0917}
N=\left\{ \begin{array}{ll} \displaystyle \sum_{j=0}^{\frac{\ell-2}{2}} b_{2j+1} & \ell \hbox{ even},\\
\min\Bigg\{b_1+\displaystyle \sum_{j=1}^{\frac{\ell-1}{2}} b_{2j}, \frac{1}{2} \displaystyle \sum_{j=1}^\ell b_j\Bigg\} & \ell \hbox{ odd}.\end{array}\right.
\end{equation}
\end{theorem}
\noindent When $\ell=2,3$ consecutive (increasing or not) and non-consecutive couplings lead to the  maximal $N$  given by \eqref{0917}. However, for $\ell\geq 4$ non-consecutive couplings may possibly give rise to a larger maximal $N$ than consecutive ones, which in turn may  do  better than the increasing consecutive coupling.

\medskip The paper is organized as follows. In Section 2 we set up  the abstract max-min scheme to provide a stable critical level $\mathcal H^*$ of $\mathcal H$. Here we make use  of a crucial compactness property which is established in Section 3. Finally, in Appendix A we derive the expression for the maximal $N$  given in Theorem \ref{maximizingN} along with a thorough discussion of the cases $\ell=2,3,4$.

\section{A max-min argument and the role of the topology of $\Sigma$} 
\setcounter{equation}{0}  
Let us outline the variational argument we are going to set up. First, we need to construct compact sets $K$, $K_0$ (with $K$ connected) and an open smooth set $\mathcal D$ so that
$$K_0\subset K\subset \mathcal D\subset \overline{\mathcal D} \subset \mathcal M,$$
where
$$\mathcal M=(\Sigma \setminus \{p_1,\dots,p_\ell\})^N \setminus \Delta, \quad \Delta=\{ \bbm[\xi] \in \Sigma^N:\ \xi_j=\xi_k \hbox{ for some }j\not= k\}.$$
Let 
$${\mathcal F}=\Big\{\Gamma(1,\cdot):\: \Gamma\in C([0,1]\times K, {\cal D})\hbox{ s.t. }
\Gamma(0,\cdot)=id_K,\: \Gamma(t,\cdot)\big|_{K_0}=id_ {K_0} \:\forall\, t\in [0,1] \Big\}$$ 
and
$$\mathcal{H}^*=\sup_{\gamma\in{\mathcal F}}\min_{\hbox{\scriptsize$\bbm[\xi]$}\in
K}\mathcal{H} (\gamma(\bbm[\xi])).$$  
Through a standard deformation argument, the existence of a critical point $\bbm[\xi]\in{\mathcal D}$ of $\mathcal{H}$ with $\mathcal{H}(\bbm[\xi])=\mathcal{H}^*$ is driven by a change in the topology of superlevel sets for $\mathcal{H}$ in ${\cal D}$ at height $\mathcal{H}^*$, as expressed by 
\begin{equation}\label{mima}\mathcal{H}^*<\min_{\hbox{\scriptsize$\bbm[\xi]$}\in K_0} \mathcal{H} (\bbm[\xi]) \end{equation}
(with the convention $ \min_{\hbox{\scriptsize$\bbm[\xi]$}\in K_0} \mathcal{H} (\bbm[\xi])=+\infty$ if $K_0=\emptyset$). To exclude the presence of constrained critical points of $\mathcal{H} \big|_{\partial \mathcal D}$ at level $\mathcal{H}^*$, we further require the following compactness condition:
\begin{equation}\label{mimabis}
\forall\: \bbm[\xi]\in \{\mathcal{H}=\mathcal{H}^*\} \cap \partial \mathcal D  \:\: \;\,\exists \tau \in T_{\hbox{\scriptsize$\bbm[\xi]$}}(\partial\mathcal
D) \hbox{ s.t. }\langle\tau,\nabla \mathcal{H}(\bbm[\xi])\rangle \neq 0,
\end{equation} where $T_{\hbox{\scriptsize$\bbm[\xi]$}}(\partial\mathcal
D) $ stands for the tangent space of $\partial \mathcal D $ at $\bbm[\xi]$.
Since properties \eqref{mima}-\eqref{mimabis} continue to hold also for functionals which are $C^1$-close to $\mathcal{H}$, notice that  such critical points are stable under $C^1$-small perturbations of $\mathcal{H}$.

\medskip Let us set
\begin{equation} \label{Hmeno}
\Phi(\bbm[\xi])= \sum_{j,k=1\atop j\neq k}^N \Gamma_j \Gamma_k G(\xi_j, \xi_k)
+\sum_{i=1}^\ell \alpha_i \sum_{j=1}^N \Gamma_j G (\xi_j,p_i)
+\sum_{j=1}^N h (\xi_j),
\end{equation}
 and for $M>0$ sufficiently large  define $\mathcal D$ as 
$${\mathcal D}=\big\{\bbm[\xi]\in{\cal M}:\, \Phi(\bbm[\xi]) <M \big\}.$$
Since
$$\Phi (\bbm[\xi])\to +\infty\hbox{ as }\bbm[\xi]\to \partial{\cal M} $$
in view of \eqref{0948} below, it follows that $\mathcal D$ is an open set with $\overline{\mathcal D}\subset \mathcal M$. Letting $\sigma_1,\dots,\sigma_N$ be (not necessarily distinct) simple, closed curves in $\Sigma \setminus \{p_1,\dots,p_\ell\}$ and $\bbm[\xi]^0=(\xi_1^0, \dots, \xi_N^0)\in \sigma_1 \times \dots \times \sigma_N$ be a $N-$tuple of distinct points, 
 introduce the  sets $K$ and $K_0$ as follows: 
\begin{eqnarray*}
&& W \hbox{ connected component of }\left\{\bbm[\xi]\in \sigma_1\times \dots \times \sigma_N :\:  \min_{j \not= k} d_g(\xi_j,\xi_k)>M^{-1}   \right\} \hbox{ s.t. }\bbm[\xi]^0 \in W\\
&&K=\overline{W},\qquad K_0=\Big\{\bbm[\xi]\in K\,\Big|\, \min_{j\neq k}d_g(\xi_j,\xi_k)=M^{-1}\Big\}
\end{eqnarray*}
for $M>( \min_{j\neq k}d_g(\xi_j^0,\xi_k^0))^{-1}$. By construction $K$ and $K_0$ are compact sets, $K$ is connected and $K_0\subset K$. 
Since $\sigma_j$, $j=1,\dots,N$, is a curve in $\Sigma \setminus \{p_1,\dots,p_\ell\}$, we have that
\beq\label{iuni}\inf\{ d_g(\xi_j,p_i): \ \xi_j \in \sigma_j, \ i=1,\dots, \ell \}>0. \eeq 
Thanks to the decomposition 
\begin{equation} \label{0948}
G(x,p)=-\frac{1}{2\pi} \log d_g(x,p)+H(x,p),\quad H \in  C(\Sigma^2),
\end{equation} 
we can rewrite $\mathcal{H}$ in \eqref{psi0} and $\Phi$ in \eqref{Hmeno} as 
\begin{equation} 
\label{1048}
\left. \begin{array}{l}
\displaystyle \mathcal{H}= \Psi_+ + O(1), \qquad \Phi= \Psi_- + O(1)\\
\displaystyle \Psi_\pm(\bbm[\xi])=
- \frac1{2\pi}\sum_{j,k=1\atop j\neq k}^N \Gamma_j \Gamma_k \log d_g(\xi_j,\xi_k) \pm \sum_{i=1}^\ell \frac{\alpha_i}{2\pi}  \sum_{j=1}^N \Gamma_j \log d_g(\xi_j,p_i),
\end{array} \right.
\end{equation}
where 
$$|(\mathcal{H}- \Psi_+)(\bbm[\xi])|+|(\Phi- \Psi_-)(\bbm[\xi])|\leq C_0 \qquad \forall \, \bbm[\xi] \in \Sigma^N$$ 
for some $C_0>0$. By \eqref{iuni} and \eqref{1048} we have that $ \sup_K \Phi  \leq C \log M$ for $M$ large, with a universal $C>0$, and then the inclusion $K\subset {\cal D}$ does hold for $M$ sufficiently large, as required.

\medskip  We are concerned now with the proof of \eqref{mima}, whereas \eqref{mimabis} will be established in the next Section thanks to the validity of \eqref{compact}. We begin with the following lemma.
\begin{lemma}\label{butta} The following holds
$$\min_{\hbox{\scriptsize$\bbm[\xi]$}\in K_0}\mathcal{H}(\bbm[\xi]) \to +\infty \hbox{ as }M \to +\infty.$$
\end{lemma}
\begin{proof}
Assume by contradiction the existence of sequences $\bbm[\xi]_n=(\xi_1^n,\dots,\xi_N^n) \in \sigma_1\times \dots \times \sigma_N$ and $M_n$ such that 
$$\sup_n \mathcal{H}(\bbm[\xi]_n)< +\infty,\quad \min_{j\neq k}d_g(\xi_j^n,\xi_k^n)=M_n^{-1} \to 0 \hbox{ as }n\to +\infty.$$ 
Up to a subsequence, we can find $j_0\neq k_0$ so that 
$$d_g(\xi_{j_0}^n,\xi_{k_0}^n) =M_n^{-1} \to 0\hbox{ as }n\to +\infty.$$ 
By \eqref{iuni} and \eqref{1048} we deduce that
$$\mathcal{H}(\bbm[\xi]_n) \geq \frac{1}{2\pi} \Gamma_{j_0} \Gamma_{k_0} \log M_n +O(1)
\to+\infty$$
as $n \to +\infty$, yielding  a contradiction. \end{proof}
\noindent Thanks to Lemma \ref{butta}, the validity of \eqref{mima} will follow once we have obtained  an upper bound on the max-min value $\mathcal{H}^*$ for $M$ large. 

\medskip To this aim, let $\mathcal{P}_j$, $j=1,\dots,N$, be a retraction of $\Sigma \setminus \{p_1,\dots,p_\ell\}$ onto $\sigma_j$, i.e. $\mathcal{P}_j: \Sigma \setminus \{p_1,\dots,p_\ell\} \to \sigma_j$ is a continuous map so that $\mathcal{P}_j  \big|_{\sigma_j}=\hbox{id}_{\sigma_j}$. A simple application of the topological degree yields  the following crucial intersection property:
\begin{theorem} \label{thm1200}
For all $\gamma \in \mathcal{F}$ there exists $\bbm[\xi]_\gamma^* \in K$ so that $\mathcal{P}_j(\gamma_j(\bbm[\xi]_\gamma^*))=\xi_j^0$ for all $j=1,\dots,N$.
\end{theorem}
\begin{proof}
Fix $\gamma \in \mathcal{F}$, and write it as $\gamma=\Gamma(1,\cdot)$, where $\Gamma\in C([0,1]\times K, {\cal D})$ satisfies $\Gamma(0,\cdot)=id_K$ and $\Gamma(t,\cdot)\big|_{K_0} =id_{K_0}$ for all $t\in [0,1]$. Extend $\Gamma$ from $K$ to $\sigma_1 \times \dots \times \sigma_N$ as $\tilde \Gamma$:
$$\tilde \Gamma(t,\bbm[\xi])=\Gamma(t,\bbm[\xi])\;\;\hbox{ if }\bbm[\xi]\in K,\;\;\;\; \tilde \Gamma(t,\bbm[\xi])=\bbm[\xi]\;\;\hbox{ if }\bbm[\xi]\in (\sigma_1 \times \dots \times \sigma_N) \setminus K.$$
Notice that $K_0$ is the topological boundary of $K$ relative to $\sigma_1 \times \dots \times \sigma_N$, and then $\tilde\Gamma \in C([0,1]\times (\sigma_1 \times \dots \times \sigma_N), \mathcal{D})$ in view of $\Gamma(t,\bbm[\xi]) =\bbm[\xi]$ for all $t\in [0,1]$, $\bbm[\xi] \in K_0$. Writing $\tilde \Gamma$ as $\tilde \Gamma=(\tilde\Gamma_1,\dots, \tilde\Gamma_N)$, the map $H:[0,1]\times  (\sigma_1 \times \dots \times \sigma_N) \to (\sigma_1 \times \dots \times \sigma_N)$ with components $H_j(t,\bbm[\xi])= (\mathcal{P}_j \circ \tilde \Gamma_j)(t,\bbm[\xi])$, $j=1,\dots,N$, is a continuous map so that $H(0,\cdot)=id_{\sigma_1\times \dots \times \sigma_N}$.

\medskip To use a degree argument, we can identify each $\sigma_j$, $j=1,\dots,N$, with $\mathbb{S}^1$ through a suitable homeomorphism, and then regard $H$ as a map $[0,1] \times (\mathbb{S}^1)^N \to (\mathbb{S}^1)^N$ with $H(0,\cdot)=id_{(\mathbb{S}^1)^N}$. Given the annulus $A=\left\{\frac12 \leq |x| \leq 2\right\}$, extend $H$ from $(\mathbb{S}^1)^N$ to $A^N$ as $\tilde H=(\tilde H_1,\dots,\tilde H_N)$, with 
$$\tilde H_j(t,x_1,\dots,x_N)=  |x_j| H_j\bigg(t,\frac{x_1}{|x_1|}, \dots, \frac{x_N}{|x_N|}\bigg),\quad (x_1,\dots,x_N) \in A^N.$$
By construction $\tilde  H$ is a continuous map from $[0,1] \times A^N$ into $A^N$, owing to $|\tilde H_j(t,x_1,\dots,x_N)|=|x_j|$ for all $t \in [0,1]$, and $\tilde H(0,\cdot)=id_{A^N}$.
Moreover,  $\tilde H(t,\cdot)$ maps the boundary $\partial( A^N )$ into itself, and we are in the position to apply a degree argument: by homotopy invariance we have that 
$$\deg(\tilde H(1, \cdot), A^N, \bbm[\xi]^0)=\deg(\tilde H(0, \cdot), A^N, \bbm[\xi]^0)=\deg (id, A^N, \bbm[\xi]^0)=1,$$
where $\bbm[\xi]^0 \in (\mathbb{S}^1)^N$ corresponds to the original $\bbm[\xi]^0 \in \sigma_1\times \dots \times \sigma_N$ through the identifications of $\sigma_j$, $j=1,\dots,N$, with $\mathbb{S}^1$. Then, there exists $x^*=(x_1^*,\dots, x_N^*)\in A^N$ so that  $\tilde H(1, x^*)=\bbm[\xi]^0$, and consequently $x^* \in (\mathbb{S}^1)^N$ thanks to $|x_j^*|=|\tilde H_j(1,x_1^*,\dots,x_N^*)|=|\xi_j^0|=1$. Getting back, we have thus found $\bbm[\xi]^* \in \sigma_1\times \dots \times \sigma_N$ so that $H(1,\bbm[\xi]^*)=\bbm[\xi]^0$. We claim that $\bbm[\xi]^* \in K$: otherwise, if $\bbm[\xi]^* \in (\sigma_1\times \dots \times\sigma_N) \setminus K$, then $H(1,\bbm[\xi]^*)=\bbm[\xi]^*$, which would lead to $\bbm[\xi]^*=\bbm[\xi]^0 $, and this provides a contradiction with $\bbm[\xi]^0\in K$. So, $\bbm[\xi]^*\in K$ and $H_j(1,\bbm[\xi]^*)=\mathcal{P}_j(\gamma_j(\bbm[\xi]^*))=\xi_j^0$ for all $j=1,\dots,N$.
\end{proof}
\noindent Since
$$\mathcal{H}^*=\sup_{\gamma\in{\mathcal F}}\min_{\hbox{\scriptsize$\bbm[\xi]$}\in
K}\mathcal{H} (\gamma(\bbm[\xi])) \leq \sup_{\gamma\in{\mathcal F}}  \mathcal{H} (\gamma(\bbm[\xi]^*_\gamma))$$
with $\bbm[\xi]^*_\gamma$ given by Theorem \ref{thm1200}, an upper bound on $\mathcal H^*$ is then reduced to show that
\begin{equation} \label{1518}
\sup_{\gamma \in \mathcal F}  \Psi_+ (\gamma(\bbm[\xi]^*_\gamma))\leq C
\end{equation}
does hold for  $M$ large, in view of  \eqref{1048}. The topological properties of $\Sigma$ play here a crucial role to find the retractions $\mathcal{P}_j$, $j=1,\dots,N$, and to investigate the structure of its fibers in order to prove \eqref{1518}. By the topological classification of closed surfaces, we have that $\Sigma$ is homeomorphic to either the sphere $\mathbb{S}^2$ or the connected sum of tori $T$ or a connected sum of real projective planes $\mathbb{RP}^2$. In the next Subsections, we will separately discuss the case $\Sigma \neq \mathbb{S}^2,\mathbb{RP}^2$ and the case $\Sigma =\mathbb{S}^2, \mathbb{RP}^2$ (up to homeomorphic equivalence), completing the proof of  Theorems \ref{mainth1}-\ref{mainth3}.

\subsection{The case $\Sigma \not= \mathbb{S}^2,\mathbb{RP}^2$} 
By Dyck's theorem (\cite{Dyck}) $\Sigma$ is homeomorphic either to the torus $T$ or to the Klein bottle or to the connected sum $T \# \Sigma'$, for a closed surface $\Sigma'$. Recall that a torus and a Klein bottle can be represented by the fundamental square $ABA^{-1}B^{-1}$ and $ABA^{-1}B$, respectively. To fix the ideas, let $A$, $B$ be a horizontal, vertical edge, respectively, and let us also assume that the singularities lie in the interior of the square. In this case, we can construct a retraction $\mathcal P$ of the surface onto $A$ by simply projecting along vertical lines, where $A$ represents a circle not passing through the singularities, and the fibers of $\mathcal P$ are well-separated:
\begin{equation} \label{fibers}
\overline{\mathcal P^{-1}(\zeta_1)} \cap \overline{\mathcal P^{-1}(\zeta_2)}=\emptyset
\end{equation}
for all $\zeta_1,\zeta_2 \in A$, $\zeta_1\neq \zeta_2$.
For $T \# \Sigma'$ the fundamental polygon looks like $ABA^{-1}\dots$ and contains three edges of a square $Q$. Let $v$ be one of the two vertices of $Q$ which do not belong to $B$. The retraction $\mathcal P$ is the projection on $A$ inside the square $Q$ and takes constant value $v$ outside $Q$, and we can still assume that $A$ does not contain any singularities. The map $\mathcal P$ is continuous and its fibers satisfy \eqref{fibers}.

\medskip Via the homeomorphism between $\Sigma$ and one of the above models, in our hands we have a retraction map $\mathcal P$ from $\Sigma$ onto a simple, closed curve $\sigma$ in $\Sigma \setminus \{ p_1,\dots,p_\ell\}$ so that
\begin{equation} \label{fibers1}
\inf\{ d_g(\xi_1,\xi_2): \  \mathcal P(\xi_1)=\mu_1, \mathcal P(\xi_2)=\mu_2 \}>0
\end{equation}
for all $\mu_1,\mu_2 \in \sigma$, $\mu_1 \neq \mu_2$, in view of \eqref{fibers}. We take $\sigma_1=\dots=\sigma_N=\sigma$, $\mathcal P_1=\dots=\mathcal P_N=\mathcal P$ and we fix $N$ distinct points $\xi_1^0,\dots,\xi_N^0 \in \sigma$. By Theorem \ref{thm1200} we have that $\gamma_j(\bbm[\xi]^*_\gamma) \in \mathcal{P}^{-1}(\xi_j^0)$ for all $j=1,\dots,N$, and then
$$\inf \big\{ d_g(\gamma_j(\bbm[\xi]^*_\gamma),\gamma_k(\bbm[\xi]^*_\gamma)): \ j \not= k, \gamma \in \mathcal F, M\geq 2(\min_{j\neq k}d_g(\xi_j^0,\xi_k^0))^{-1} \big\}>0$$
does hold thanks to  \eqref{fibers1}, yielding  the validity of \eqref{1518}.

\subsection{The case $\Sigma=\mathbb{S}^2, \mathbb{RP}^2$} \label{subs} 
Since $\Sigma$ is a smooth surface, there exists (see for example \cite{Hirsch}) a diffeomorphism $\omega$ from $\Sigma$ to $\mathbb{S}^2$ or $\mathbb{RP}^2$. Let $\tilde p_1=\omega(p_1),\dots,\tilde p_\ell=\omega(p_\ell)$ be the corresponding singular  sources, and $\tilde g=(\omega^{-1})* g$ be the induced metric. In this way, $(\Sigma,g)$ is isometrically equivalent to $\mathbb{S}^2$ or $\mathbb{RP}^2$ endowed with the metric $\tilde g$. In Theorem \ref{mainth3} let us first consider a consecutive coupling where each $p_i$ is coupled with $p_{i+1}$ ($p_{\ell+1}:=p_1$), i.e. $J_1=\{\ell \}$ and $J_i=\{i-1\}$ for all $i=2,\dots,\ell$. The argument is already involved and contains the main ideas of the general case, which will be discussed in a sketched way right after.

\medskip   By using $\tilde p_\ell$ as north pole on $\mathbb{S}^2$,
we can construct a diffeomorphism $\Pi: \Sigma \setminus \{p_\ell\} \to \mathbb{C}$ so that $\Pi(p_i)= q_i \in \mathbb{R}$ for all $i=1,\dots,\ell-1$, with $q_i=i$,  and $\Pi \circ \omega^{-1}$ coincides with the stereographic projection through $\tilde p_\ell$ in a small neighborhood of $\tilde p_\ell$. Since 
$$ \frac{1}{C} \, g_0 \leq \tilde g \leq C \,  g_0$$
does hold in every coordinate open set $\Omega \subset \mathbb{S}^2$ for some $C=C(\Omega)>1$, where $g_0$ is the round metric on $\mathbb{S}^2$, by compactness of $\mathbb{S}^2$ we get
$$\frac{1}{C} \leq \frac{d_{g_0}(x,y)}{d_{\tilde g}(x,y)}\leq C \quad \forall \, x,y \in \mathbb{S}^2$$
for some $C>1$. Since $(\mathbb{S}^2 \setminus \{\tilde p_\ell\},g_0)$ is isometrically equivalent to $(\mathbb{C},g_1)$, $g_1=\frac{4}{(1+|z|^2)^2} dx dy$ $(z=x+iy)$, via the stereographic projection, we have that there exists $C>1$ so that
\begin{equation} \label{comparable}
\frac{1}{C} \leq \frac{d_g(x,y)}{d_{g_1}(\Pi(x),\Pi(y))} \leq C \quad \forall\, x,y \in \Sigma \setminus \{p_\ell\}.
\end{equation}
Indeed, \eqref{comparable} is true on compact subsets of $\Sigma \setminus
\{p_\ell\}$, while near $p_\ell$ it follows by the property that $\Pi \circ \omega^{-1}$ coincides with the stereographic projection through $\tilde p_\ell$ near $\tilde p_\ell$. 

\medskip  \medskip  Thanks to \eqref{comparable}, we can work directly in $\mathbb{C}$. Let us now define a continuous map $ \Upsilon_{i,r}:\mathbb{C}\setminus\{q_{i},q_{r}\}\to \sphe^1$, $i \not=r$, as follows: 
$$\Upsilon_{i,r}(z)=\left\{\begin{array}{ll}
e^{{\rm i}\arg(z-q_{i})}&\hbox{if } \Ree \, z \leq \frac{i+r}2\\ 
e^{{\rm i} (\pi-\arg(z-q_{r}))}&\hbox{if } \Ree \, z \geq \frac{i+r}2
 \end{array}\right.  $$
if $i < r\leq \ell-1$, $\Upsilon_{i,\ell}(z)=e^{{\rm i}\arg(z-q_i)}$ if $i<r=\ell$ and $\Upsilon_{i,r}=\frac{1}{\Upsilon_{r,i}}$ if $r < i$. For $\theta\in (-\frac{\pi}{2},\frac{\pi}2)$ notice that the fibers $L_{i,r}(\theta)=\Upsilon_{i,r}^{-1}( e^{{\rm i}\theta })$ represent
\begin{itemize}
\item the two vertical edges of the isosceles triangle with base $\overline{q_{i}q_{r}}$ and base angle $\theta$ 
$$L_{i,r}(\theta)= \left\{ q_{i}+\rho e^{{\rm i}\theta}\,\Big|\,0< \rho\leq \frac{r-i}{2 \cos\theta}  \right\}\cup\left\{ q_{r}-\rho e^{-{\rm i}\theta}\,\Big|\, 0< \rho\leq \frac{r-i}{2\cos\theta} \right\}$$
for $i<r\leq \ell-1$;
\item the straight line starting from $q_i$ with angle $\theta$ for $i<r=\ell$;
\item the set $L_{r,i}(-\theta)$ for $r<i$.
\end{itemize}

\medskip  Split $\{1,\dots,N\}$ as the disjoint union of $\mathcal{N}_i$, $i=1,\dots,\ell$, and define continuous maps $\mathcal{P}_j:\Sigma \setminus \{p_1,\dots,p_\ell\} \to \sigma_j$, $j=1,\dots,N$, as  
$$\mathcal{P}_j= \Pi^{-1} \left( q_i+\frac{1}{4} \Upsilon_{i,i+1} \circ \Pi \right) ,\:\: \sigma_j=\Pi^{-1}\left(q_i+\frac{1}{4} \mathbb{S}^1\right)$$
when $j \in \mathcal{N}_i$, $ i=1,\dots,\ell-1$, and 
$$\mathcal{P}_j= \Pi^{-1} \left(q_1+\ell \Upsilon_{1,\ell} \circ \Pi  \right) ,\:\: \sigma_j=\Pi^{-1}\left(q_1+\ell \mathbb{S}^1\right)$$
when $j \in \mathcal{N}_\ell$. Notice that $\mathcal{P}_j  \big|_{\sigma_j}=\hbox{id}_{\sigma_j}$ for all $j=1,\dots,N$. Let us fix $N$ distinct angles $\theta_1,\dots,\theta_N\in (0,\frac{\pi}{2})$, and let $\bbm[\xi]^0=(\xi_1^0, \dots, \xi_N^0)\in \sigma_1 \times \dots \times \sigma_N$ be a $N-$tuple of distinct points, with 
$$\xi_j^0= \left\{ \begin{array}{ll} \Pi^{-1}(q_i+\frac{1}{4} e^{{\rm i}\theta_j})& \hbox{if }j \in \mathcal{N}_i, \, i=1,\dots,\ell-1\\
\Pi^{-1}(q_1+\ell e^{-{\rm i}\theta_j})& \hbox{if }j \in \mathcal{N}_\ell. \end{array} \right.$$
Thanks to Theorem \ref{thm1200}, for all $\gamma \in \mathcal{F}$ we can find $\bbm[z]^\gamma =(z^\gamma_1,\dots,z^\gamma_N)\in \mathbb{C}^{N}$, with $z^\gamma_j= \Pi [\gamma_j (\bbm[\xi]_\gamma^*)]$ for $j=1,\dots,N$, so that 
\begin{equation*}
z^\gamma_j \in L_{i,i+1}(\theta_j),\quad \forall\, j \in \mathcal{N}_i, \, i=1,\dots,\ell
\end{equation*}
(with $\ell+1=1$). In view of \eqref{comparable},  \eqref{1518} can be re-formulated  as
\begin{equation} \label{1341}
\sup_{\gamma \in \mathcal F}  \Psi (\bbm[ z]^\gamma)\leq C,
\end{equation}
where 
\begin{equation} \label{1528}
\Psi(\bbm[z])=- \frac1{2\pi}\sum_{j,k=1\atop j\neq k}^N \Gamma_j \Gamma_k \log d_{g_1}(z_j,z_k)+\sum_{i=1}^\ell \frac{\alpha_i}{2\pi}  \sum_{j=1}^N \Gamma_j \log d_{g_1}(z_j,q_i)
\end{equation}
for $\bbm[z]=(z_1,\dots,z_N) \in \mathbb{C}^{N}$, with $q_\ell=\infty$. Next lemma establishes 
the validity of \eqref{1341}. 
\begin{lemma}\label{prevenar} 
If $\theta_1,\dots,\theta_N\in (0,\frac{\pi}{2})$ are distinct angles, then
$$\sup_{\bbm[z] \in Z} \Psi(\bbm[z])<+\infty,$$
where $\Psi$ is given by \eqref{1528} and
$$Z=\Big\{\bbm[z]=(z_1,\dots,z_N) \in \mathbb{C}^{N}: \; z_j \in L_{i,i+1}(\theta_j) \;\, \forall \, j \in \mathcal{N}_i,\, i=1,\dots,\ell \Big\}.$$
\end{lemma}
\begin{proof} 
Observe that
\begin{equation} \label{1814}
\overline{L_{i,i+1}(\theta_j)} \cap \overline{L_{r,r+1}(\theta_k)}=\left\{ \begin{array}{ll}  q_i,q_{i+1} &\hbox{if }r=i\\
q_{i+1} &\hbox{if }r=i+1 \\
\emptyset &\hbox{otherwise}\end{array} \right.
\end{equation}
for all $j \in \mathcal N_i$ and $k\in \mathcal N_r$ with $j \neq k$, where $q_{\ell+1}=q_1$ and the closure is meant with respect to $d_{g_1}$. By \eqref{1814} we can rewrite $\Psi$ as
\begin{equation*}
\Psi(\bbm[z]) =
-\frac{1}{2 \pi}\sum_{i=1}^{\ell}\sum_{j \in \mathcal{N}_i}\sum_{
k \in \mathcal{N}_{i-1} \cup \mathcal{N}_i \cup \mathcal{N}_{i+1}\atop k\neq j}
\Gamma_j \Gamma_k \log d_{g_1}(z_j,z_k)+\sum_{i=1}^\ell \frac{\alpha_{i}}{2 \pi} \sum_{j\in 	\tilde{\mathcal N}_i} \Gamma_j \log d_{g_1}(z_j,q_i)+O(1),
\end{equation*}
where $\tilde{\mathcal N}_i=\mathcal N_i \cup \mathcal N_{i-1}$ for a consecutive coupling, $\mathcal N_0=\mathcal N_\ell$, $\mathcal N_{\ell+1}=\mathcal N_1$ and $\alpha_{\ell+1}=\alpha_1$.

\medskip  For $i=1,\dots,\ell$ and $j,k \in \mathcal N_i$, $j\neq k$, the two  polygonals $L_{i,i+1}(\theta_j)$ and $L_{i,i+1}(\theta_k)$ approach the same end-points $q_{i}$ and  $q_{i+1}$ with different angles $\theta_j\neq\theta_k$, yielding  the inequality
\begin{equation} \label{1848}
d_{g_1}(z_j,z_k)\geq \delta \, d_{g_1}(z_j,q_i) \, d_{g_1}(z_j,q_{i+1})
\end{equation}
for all $z_j \in L_{i,i+1}(\theta_j)$ and $z_k \in L_{i,i+1}(\theta_k)$, where $\delta>0$ depends only on $\theta_1,\dots ,\theta_N$. Indeed,  for $i=1,\dots, \ell-2$ we can write points $z_j$, $z_k$ near $q_i$ as $z_j=q_i+|z_j-q_i|e^{{\rm i}\theta_j}$, $z_k=q_i+|z_k-q_i|e^{{\rm i}\theta_k}$ to get
$$|z_j-z_k|^2=|z_j-q_i|^2-2 |z_j-q_i||z_k-q_i| \cos |\theta_k-\theta_j|+|z_k-q_i|^2  \geq |z_j-q_i|^2 \sin^2 |\theta_k-\theta_j|,$$
and  \eqref{1848} follows near $q_i$ owing to the equivalence between $d_{g_1}$ and the euclidean distance on compact subsets of $\mathbb{R}^2$. A similar argument works near $q_{i+1}$, and \eqref{1848} is thus proved when $i=1,\dots,\ell-2$. Inequality \eqref{1848} is still valid near $q_i$ for $i=\ell-1$ and near $q_{i+1}$ for $i=\ell$, and the difficult case is when approaching $q_\ell=\infty$. For $z_j=q_{\ell-1}+|z_j-q_{\ell-1}|e^{{\rm i}\theta_j}$ and $z_k=q_{\ell-1}+|z_k-q_{\ell-1}|e^{{\rm i}\theta_k}$ the following  holds
\begin{eqnarray*}
\Big|\frac{1}{z_j}-\frac{1}{z_k}\Big|^2 &=& \frac{|z_j-z_k|^2}{|z_j|^2|z_k|^2}
=\frac{|z_j-q_{\ell-1}|^2-2 |z_j-q_{\ell-1}||z_k-q_{\ell-1}| \cos |\theta_k-\theta_j|+|z_k-q_{\ell-1}|^2}{|z_j|^2|q_{\ell-1}+|z_k-q_{\ell-1}| e^{{\rm i} \theta_k}|^2}\\
& \geq &
\frac{|z_k-q_{\ell-1}|^2 \sin^2 |\theta_k-\theta_j|}{2 |z_j|^2 |z_k-q_{\ell-1}|^2}\geq 
\delta\Big|\frac{1}{z_j}\Big|^2,
\end{eqnarray*}
for $|z_j|,|z_k|$ large,  providing the validity of \eqref{1848} with $i=\ell-1$ near $q_\ell=\infty$ in view of the invariance of $g_1$ under the map $z \to \frac{1}{z}$ and the equivalence between $d_{g_1}$ and the euclidean distance near $0$. A similar argument works for \eqref{1848} with $i=\ell$ near $q_\ell=\infty$, and \eqref{1848} is finally established for $i=1,\dots,\ell$.

\medskip  For $i=1,\dots,\ell$ and $j\in \mathcal N_{i}$, $k\in \mathcal N_{i+1}$, the two polygonals $L_{i,i+1}(\theta_j)$ and $L_{i+1,i+2}(\theta_k)$ (with $L_{\ell+1,\ell+2}(\theta_k)=L_{1,2}(\theta_k)$) just have $q_{i+1}$ as common end-point, and, arguing as above, we deduce
\begin{equation} \label{2208}
d_{g_1}(z_j,z_k)\geq \delta \max\{ d_{g_1}(z_j,q_{i+1}) \, d_{g_1}(z_k,q_{i+1})\}
\end{equation}
for all $z_j \in L_{i,i+1}(\theta_j)$ and $z_k \in L_{i+1,i+2}(\theta_k)$, where $\delta>0$ just depends on $\theta_1,\dots,\theta_N$. 

\medskip  For $i=1,\dots, \ell$ and  $j\in \mathcal N_i$, by \eqref{1848} and \eqref{2208} we deduce 
that
\begin{eqnarray*}
&&-\sum_{k\in \mathcal N_{i-1}\cup \mathcal N_{i}\cup \mathcal N_{i+1}\atop k\neq j} \Gamma_k \log d_{g_1}(z_j,z_k)\\
&&=- \sum_{
k\in \mathcal N_{i-1}} \Gamma_k  \log d_{g_1}(z_j,z_k) -\sum_{k\in \mathcal N_{i}\atop k\neq j} \Gamma_k \log d_{g_1}(z_j,z_k)- \sum_{k\in \mathcal N_{i+1}} \Gamma_k  \log d_{g_1}(z_j,z_k)
\\ 
&&\leq -\sum_{
k\in \mathcal N_{i-1}} \Gamma_k \log d_{g_1}(z_j,q_i)- \sum_{
k\in \mathcal N_{i}\atop k\neq j} \Gamma_k  \log d_{g_1}(z_j,q_i)-\sum_{
k\in \mathcal N_{i}\atop k\neq j} \Gamma_k \log d_{g_1}(z_j,q_{i+1}) \\
&&\quad - \sum_{
k\in  \mathcal N_{i+1}} \Gamma_k \log d_{g_1}(z_j,q_{i+1})+O(1)\\
&&=- \sum_{
k\in \tilde{\mathcal N}_{i} \atop k\neq j} \Gamma_k  \log d_{g_1}(z_j,q_i)
- \sum_{
k\in  \tilde{\mathcal N}_{i+1} \atop k \neq j} \Gamma_k \log d_{g_1}(z_j,q_{i+1})
+O(1).
\end{eqnarray*}
Therefore, we have shown that
\begin{eqnarray*}
\Psi(\bbm[z]) &=& 
- \frac{1}{2 \pi}\sum_{i=1}^{\ell}\sum_{j \in \mathcal{N}_i} \sum_{
k\in \tilde{\mathcal N}_{i} \atop k\neq j} \Gamma_j \Gamma_k  \log d_{g_1}(z_j,q_i)
- \frac{1}{2 \pi}\sum_{i=1}^{\ell}\sum_{j \in \mathcal{N}_i}
 \sum_{
k\in  \tilde{\mathcal N}_{i+1} \atop k \neq j} \Gamma_j  \Gamma_k \log d_{g_1}(z_j,q_{i+1})\\
&&\quad +\sum_{i=1}^\ell \frac{\alpha_{i}}{2 \pi} \sum_{j \in \tilde{\mathcal N}_{i}} \Gamma_j \log d_{g_1}(z_j,q_i)+
O(1)\\
&=& 
- \frac{1}{2 \pi}\sum_{i=1}^{\ell}\bigg[ \sum_{j,k \in \tilde{\mathcal N}_{i} \atop k\neq j} \Gamma_j \Gamma_k  - \alpha_i \sum_{j \in \tilde{\mathcal N}_{i}} \Gamma_j\bigg] \log d_{g_1}(z_j,q_i)+O(1)
\end{eqnarray*}
and the above quantity is uniformly bounded from above with respect to  $\bbm[z] \in Z$ in view of \eqref{0956}, yielding  $\sup_Z \Psi<+\infty$.
\end{proof}

Letting $\sim$ be the equivalence relation between antipodal points, the surface $\mathbb{RP}^2$ can be represented as the quotient $\mathbb{S}^2 / \hspace{-0.1cm} \sim$. 
We can find a retraction $\mathcal P_j: \mathbb{RP}^2 \setminus \{ \tilde p_i\} \to C$, $j \in \mathcal N_i$, as the projection along great circles passing through $\tilde p_i$ onto a given equatorial circle $C$ not passing through $\tilde p_1,\dots,\tilde p_\ell$. The fibers of $\mathcal P_j$, $j \in \mathcal N_i$, intersect in $\tilde p_i$, a fact which can be controlled by an assumption like \eqref{1530b}. However the fibers of $\mathcal P_j$ and $\mathcal P_k$ have intersection points outside $\tilde p_1,\dots,\tilde p_\ell$, for all $j \in \mathcal N_{i_1}$ and $k \in \mathcal N_{i_2}$ with $i_1\neq i_2$, and an upper bound on $\mathcal H^*$ is not generally available. In Theorem \ref{mainth2} we then restrict the attention to the special case $\mathcal N_i=\{1,\dots,N\}$, with $\alpha_i=\max\{\alpha_1,\dots, \alpha_\ell \}$ for some $i=1,\ldots,\ell$. Taking $\tilde p_i$ as the north pole and $C$ as the equator, we have that the upper hemisphere $\mathbb{S}^2_+$ (w.r.t. $C$) can be projected onto the equatorial plane and the unit disc $D$ with identified antipodal boundary points is a model for $\mathbb{RP}^2$. 
Then we can find  a diffeomorphism $\Pi: \Sigma \to D$ so that $\Pi(p_r)=q_r \in \mathbb{R}$, $r=1,\dots,\ell$, with $q_i=0$ and $-1\leq q_1<\dots<q_\ell\leq1$. Moreover, by compactness of $\Sigma$ the following holds 
\begin{equation} \label{comparablebis}
\frac{1}{C} \leq \frac{d_g(x,y)}{|\Pi(x)-\Pi(y)|} \leq C \quad \forall\, x,y \in \Sigma 
\end{equation}
for some $C>1$. Letting $\pi: D \setminus \{0 \} \to \partial D$ be the radial projection, we define $\mathcal{P}_j=\mathcal P$ for all $j=1,\dots,N$, where $\mathcal P=\Pi^{-1}\circ \pi \circ \Pi:\Sigma \setminus \{p_i\} \to \sigma$ and $\sigma=\Pi^{-1}(\partial D) $. Let us fix $N$ distinct points $\xi_1^0, \dots, \xi_N^0 \in \sigma \setminus \{p_1,\dots,p_\ell\}$. In view of \eqref{comparablebis}, the validity of \eqref{1518} will follow by
\begin{equation} \label{1752}
\sup_{\bbm[z] \in Z} \Psi(\bbm[z])<+\infty,\quad \Psi(\bbm[z])=- \frac1{2\pi}\sum_{j,k=1\atop j\neq k}^N \Gamma_j \Gamma_k \log |z_j-z_k|+\sum_{i=1}^\ell \frac{\alpha_i}{2\pi}  \sum_{j=1}^N \Gamma_j \log |z_j-q_i|,
\end{equation}
where 
$$Z=\Big\{\bbm[z]=(z_1,\dots,z_N) \in D^{N}: \, z_j \in \mathcal \pi^{-1}(\zeta_j^0)\;\; \forall \, j=1,\dots,N\Big\}, \quad \zeta_j^0=\Pi(\xi_j^0).$$
Since $\overline{\pi^{-1}(\zeta_j^0)} \cap \overline{\pi^{-1}(\zeta_k^0)}=\{0 \}$ for $j \neq k$, arguing as in Lemma \ref{prevenar} \eqref{1752} can be deduced by \eqref{1530}. 

\medskip For a general coupling in Theorem \ref{mainth3}, we explain below the necessary changes.
Denoting by $r(i)$  the unique index such that $i\in J_{r(i)}$, $i\in\{1,\dots,\ell\}$, we split
$$\{1,\dots,\ell\}=\bigcup_{r=1}^m X_r$$ 
in disjoint blocks $X_r$ which satisfy 
\beq\label{minimal} i\in X_r\Rightarrow J_i\subset X_r  \;\; \&\;\; r(i) \in X_r\eeq 
and are minimal (i.e. no proper subset of $X_r$ satisfies \eqref{minimal}). Notice that such a partition  $\{X_1,\dots, X_m\}$ is unique, and \eqref{minimal} guarantees that there are no couplings between indices in different blocks. Thanks to the following result, we can provide each block with a nice order of all the indices but one (say, the last one):
\begin{lemma}\label{permit}Let $X=\{x_1,\dots,x_n\} \subset \mathbb{R}$ be a set of $n\geq 2$ elements. Let $J_1,\dots,J_n$ be a partition of $X$ so that $x_i \notin J_i$, $i=1,\dots,n$, and 
\beq\label{minimalbis} X' \subset X:\,  J_i \subset X'  \;\; \&\;\; x_{r(i)} \in X' \: \: \forall \, x_i \in X'  \Longrightarrow X'=X, \eeq 
where $r(i) \in \{1,\dots,n\}$ is the unique index so that $x_i \in J_{r(i)}$. Then, there exists a permutation $\sigma$ of $\{1,\ldots,n\}$   so that\footnote{Hereafter, the notation $A<B$ ($A\leq B$ resp.) stands for $\sup A<\inf B$   ($\sup A\leq \inf B$ resp.) if $A,B\neq \emptyset$. The inequality is always true if either $A=\emptyset$ or $B=\emptyset$, and points are identified with singletons.}
\beq\label{per1} 
\sigma(i)<\sigma(j)  \Rightarrow \hbox{ either } J_{\sigma(i)}^*<x_{\sigma(i)} \leq J_{\sigma(j)}^*<x_{\sigma(j)} \hbox{ or }J_{\sigma(j)}^*<J_{\sigma(i)}^*<x_{\sigma(i)}<x_{\sigma(j)}, \eeq 
where  $J_i^*=J_i\setminus \{x_{\sigma(n)} \}$, $i=1,\dots,n$.
\end{lemma}
\begin{proof} 
We argue by induction on $n$. When $n=2$ we have that $J_1=\{x_2\}$, $J_2= \{x_1\}$ and \eqref{per1} is satisfied with $\sigma=\hbox{id}$. If Lemma \ref{permit} does hold for $n$, let us discuss its validity also for $n+1$. If $J_i\neq \emptyset $ for any $i=1,\dots,n+1$, then $\# J_i=1$ for all $i$, and by \eqref{minimalbis} we can find a permutation $\sigma $ so that $J_{\sigma(i)}=\{x_{\sigma(i-1)} \}$ for $i \geq 2$ and  $J_{\sigma(1)}=\{x_{\sigma(n)} \}$, and \eqref{per1} easily follows. Otherwise, we can find a first permutation $\tau $ so that $J_{\tau(n+1)}=\emptyset$ and $\tau(n)=r(\tau(n+1))$. Letting $Y=\{x_{\tau(1)},\dots, x_{\tau(n)} \}$, we have that $\# Y=n$ and  $Y$ still satisfies
\eqref{minimalbis} with the partition $J_{\tau(1)},\dots, J_{\tau(n-1)}, J_{\tau(n)} \setminus \{x_{\tau(n+1)} \}$. Since Lemma \ref{permit} is true for $Y$, we can find a permutation $\lambda$ of $\{1,\ldots,n\}$ so that \eqref{per1} does hold for $Y$ with $\lambda$. The permutation $\sigma $ of $\{1,\ldots,n+1\}$ constructed as $\sigma(1)= \tau \circ \lambda (1),\dots, \sigma(i_0-1)=\tau \circ \lambda  (i_0-1), \sigma(i_0)=\tau(n+1), \sigma(i_0+1)=\tau \circ \lambda  (i_0),\dots,\sigma(n+1)=\tau \circ \lambda (n)$ with $i_0$ defined by $\lambda(i_0)=\tau(n)$, satisfies  \eqref{per1}, as it can be straightforwardly checked.
\end{proof}
We first make a permutation to have $X_1<\dots<X_m$ and then apply Lemma \ref{permit} to each $X_r$, $r=1,\dots,m$, to get the following:
\begin{proposition} \label{permitprop}
Up to a permutation, there exist blocks $X_1<\dots<X_m$, $m\geq 1$, satisfying \eqref{minimal} and for all $r=1,\dots,m$:
\begin{eqnarray} \label{2313} 
 i,j \in X_r,\, i<j   \Rightarrow \hbox{ either } J_i^*<i \leq J_j^*<j \hbox{ or }J_j^*<J_i^*<i<j,
\end{eqnarray}
where  $J_i^*=J_i\setminus \{l_r\}$ and $l_r=\max X_r$.
\end{proposition}
Hereafter, let us assume that we have permuted the indices according to Proposition \ref{permitprop}. Define continuous maps $\mathcal{P}_j:\Sigma \setminus \{p_1,\dots,p_\ell\} \to \sigma_j$, $j=1,\dots,N$, as  
$$\mathcal{P}_j= \Pi^{-1} \left( q_i+\frac{1}{4} \Upsilon_{i,r(i)} \circ \Pi \right) ,\:\: \sigma_j=\Pi^{-1}\left(q_i+\frac{1}{4} \mathbb{S}^1\right)$$
when $j \in \mathcal{N}_i$, $ i=1,\dots,\ell-1$, and 
$$\mathcal{P}_j= \Pi^{-1} \left(q_{r(\ell)}+ \ell \Upsilon_{r(\ell),\ell} \circ \Pi  \right) ,\:\: \sigma_j=\Pi^{-1}\left(q_{r(\ell)}+\ell \mathbb{S}^1\right)$$
when $j \in \mathcal{N}_\ell$. Let us fix $N$ angles 
\begin{eqnarray} \label{angles}
0<\theta_N<\dots<\theta_1<\frac{\pi}{2}, 
\end{eqnarray}
and let $\bbm[\xi]^0=(\xi_1^0, \dots, \xi_N^0)\in \sigma_1 \times \dots \times \sigma_N$ be a $N-$tuple of distinct points, with 
$$\xi_j^0= \left\{ \begin{array}{ll} \Pi^{-1}(q_i+\frac{1}{4} e^{{\rm i}\theta_j})& \hbox{if }j \in \mathcal{N}_i, \, i=1,\dots,\ell-1,\\
\Pi^{-1}(q_{r(\ell)}+\ell e^{-{\rm i}\theta_j})& \hbox{if }j \in \mathcal{N}_\ell. \end{array} \right.$$
Thanks to Theorem \ref{thm1200}, for all $\gamma \in \mathcal{F}$ we can find $\bbm[z]^\gamma =(z^\gamma_1,\dots,z^\gamma_N)\in Z$ (with $z^\gamma_j= \Pi [\gamma_j (\bbm[\xi]_\gamma^*)]$ for $j=1,\dots,N$), where
$$Z=\{\bbm[z]=(z_1,\dots,z_N) \in \mathbb{C}^{N}: \, z_j \in L_{i,r(i)}(\theta_j) \;\, \forall  j \in \mathcal{N}_i,\, i=1,\dots,\ell \}.$$
Arguing as in Lemma \ref{prevenar}, the aim now is to discuss the set $\overline{L_{i,r(i)}(\theta_j)} \cap \overline{L_{s,r(s)}(\theta_k)}$ for $r(i) \leq r(s)$ and $j \in \mathcal N_i$, $k\in \mathcal N_s$ with $j \neq k$, 
where the closure is meant with respect to $d_{g_1}$. Since $i \in J_{r(i)}$, by \eqref{minimal} notice that $i \in X_r$ if and only if $r(i) \in X_r$, and by \eqref{2313}-\eqref{angles} the following distinct alternatives can arise:
\begin{itemize}
\item if $i=s$, then $\overline{L_{i,r(i)}(\theta_j)} \cap \overline{L_{s,r(s)}(\theta_k)}=\{q_i,q_{r(i)} \}$ (with $q_\ell=\infty$);
\item if $i \not=s$ and $r(i)=r(s)$, then $\overline{L_{i,r(i)}(\theta_j)} \cap \overline{L_{s,r(s)}(\theta_k)}=\{ q_{r(i)} \}$;
\item if $i,s \in X_r$ with $r(i)<r(s)$ and $i=l_r$, then $s<r(s)$ and
$$\overline{L_{i,r(i)}(\theta_j)} \cap \overline{L_{s,r(s)}(\theta_k)}=\left\{ \begin{array}{ll}  \{q_s,q_i\} &\hbox{if }r(i)=s,\, r(s)=l_r  \\
\{q_i \} &\hbox{if }r(i) \neq s,\, r(s)=l_r \\
\{q_s\} &\hbox{if }r(i)=s,\, r(s) \neq l_r \\
\emptyset &\hbox{if }r(i)\neq s,\, r(s)\neq l_r; \end{array} \right. $$
\item if $i,s \in X_r$ with $r(i)<r(s)$ and $s=l_r$, then $i<r(i)<r(s)<s$ and $\overline{L_{i,r(i)}(\theta_j)} \cap \overline{L_{s,r(s)}(\theta_k)}=\emptyset$;
\item if $i,s \in X_r \setminus \{l_r\}$ with $r(i)<r(s)$, then either $i<r(i)\leq s <r(s)$ with
$$\overline{L_{i,r(i)}(\theta_j)} \cap \overline{L_{s,r(s)}(\theta_k)}=\left\{ \begin{array}{ll}  \{q_s\} &\hbox{if }r(i)=s \\
\emptyset &\hbox{if }r(i)\neq s, \end{array} \right. $$
or $s<i<r(i)<r(s)$ with $\overline{L_{i,r(i)}(\theta_j)} \cap \overline{L_{s,r(s)}(\theta_k)}=\emptyset$;
\item if $i$ and $s$ belong to different blocks, then $\overline{L_{i,r(i)}(\theta_j)} \cap \overline{L_{s,r(s)}(\theta_k)}=\emptyset$.
\end{itemize}
Since the $L_{i,r(i)}(\theta_j)'$s can share at most endpoints among $p_1,\dots,p_\ell$, we just need to analyze the behavior at each $p_i$. Every $p_i$ is an endpoint 
of $L_{i,r(i)}(\theta_j)$, $j \in \mathcal N_i$, and of $L_{s,r(s)}(\theta_k)$, $s \in J_i$ and $k \in \mathcal N_s$. Letting 
$$\tilde{\mathcal N}_i=\mathcal N_i \cup \bigcup_{r \in J_i} \mathcal N_r,$$
we can argue precisely as in Lemma \ref{prevenar} to show  that \eqref{0956} implies
$$\sup_{\bbm[z] \in Z} \Psi(\bbm[z])<+\infty$$
with $\Psi$ given in \eqref{1528}, which in turn is equivalent to \eqref{1518}.

\section{A compactness property} 
\setcounter{equation}{0}  
We shall show that \eqref{mimabis} holds provided that $M$ is sufficiently large. Since the choice $\gamma=id_K$ in the definition of $\mathcal H^*$ leads to
$$ \mathcal H^* \geq\min_{\hbox{\scriptsize$\bbm[\xi]$}\in K} \Psi_+ (\bbm[\xi])+O(1)\geq 
\frac{1}{2\pi}   \bigg(\sum_{i=1}^\ell \alpha_i\bigg) \bigg(\sum_{j=1}^N \Gamma_j\bigg) \log d
-\frac{1}{2\pi} \bigg(\sum_{j,k=1\atop j\neq k}^N \Gamma_j \Gamma_k \bigg)\log \hbox{diam}\,\Sigma 
+O(1) $$
in view of \eqref{1048}, where $d=\displaystyle \inf \{d_g(\xi,p_i):\ \xi \in \sigma_j,\,i=1,\dots,\ell,j=1,\dots,N\}>0$, by the previous Section we deduce that $\cal H^*$ is uniformly bounded in $M$. Therefore, it is enough to show that the tangential derivative of $\mathcal H$ on $\partial \mathcal D$ is non-zero for uniformly bounded values of $\cal H$ when $M$ is large enough. By contradiction assume that there exist $\bbm[\xi]_n=(\xi_1^n, \dots, \xi_N^n)\in {\cal M}$ and $(\beta^n_1, \beta_2^n)\neq (0,0)$ such that 
\begin{eqnarray}
 &\Phi(\bbm[\xi]_n)\to +\infty, \qquad
-C\leq \cal H(\bbm[\xi]_n)\leq C,& \label{boupsi} \\
 & \beta_1^n\nabla \cal H(\bbm[\xi]_n)+\beta_2^n\nabla \Phi(\bbm[\xi]_n)=0& \nonumber
\end{eqnarray}
for some $C>0$, where the last expression accounts also for non-regular points of $\partial \mathcal D$ and can  be re-written as
\beq\label{leone}
(\beta_2^n-\beta_1^n)\Gamma_j\sum_{i=1}^\ell\alpha_i\nabla_{\xi_j} G(\xi_j^n,p_i)+2(\beta_1^n+\beta_2^n)\Gamma_j\sum_{k=1\atop k\neq j}^N\Gamma_k\nabla_{\xi_j} G(\xi_j^n,\xi_k^n)=O(1)\qquad \forall \ j.\eeq
To get a contradiction, our aim is to identify the leading term of the left hand side in \eqref{leone}.
Without loss of generality we assume that
\beq\label{ops}(\beta_1^n)^2+(\beta_2^n)^2=1,\quad \beta_1^n-\beta_2^n\geq 0, \quad \min_{j\neq k}d_g(\xi_j^n,\xi_k^n)=o(1),\quad \min_{j,i}d_g(\xi_j^n,p_i)=o(1), \eeq
where we have used
$$2\sum_{j,k=1\atop j\neq k}^N\Gamma_j\Gamma_kG(\xi_j^n, \xi_k^n)=\cal H(\bbm[\xi]_n)+\Phi(\bbm[\xi]_n)+O(1)\to +\infty$$
and $$2\sum_{i=1}^\ell \alpha_i\sum_{j=1}^N\Gamma_j G(\xi_j, p_i)= \Phi(\bbm[\xi]_n)-\cal H(\bbm[\xi]_n)+O(1)\to +\infty$$
thanks to \eqref{boupsi}. Given $r_0>0$ small enough (smaller than the injectivity radius of $(\Sigma,g)$), we introduce normal coordinates $y_\xi: y_\xi^{-1}(B_{r_0}(0)) \to B_{r_0}(0)$ which depend smoothly on $\xi \in \Sigma$. Since $y_\xi(\xi)=0$ and $d_g(x,\xi)=|y_\xi(x)|$ for all $x \in y_\xi^{-1}(B_{r_0}(0))$, we have that
\begin{equation} \label{1538}
\nabla_{\xi_1} \log d_g(\xi_1,\xi_2)= \frac{y_{\xi_2}(\xi_1)}{|y_{\xi_2}(\xi_1)|^2}=
\frac{y_\xi(\xi_1)-y_\xi(\xi_2)}{|y_\xi(\xi_1)-y_\xi(\xi_2)|^2}+o\Big(\frac{1}{d_g(\xi_1,\xi_2)} \Big)
\end{equation}
as $\xi_1,\xi_2 \to \xi$, owing  to
\begin{eqnarray*}
y_\xi(\xi_1)-y_\xi(\xi_2)&=&y_{\xi_2}(\xi_1)-y_{\xi_2}(\xi_2)+O(d_g(\xi_2,\xi))|\nabla_\xi y_{\tilde \xi}(\xi_1)-\nabla_\xi y_{\tilde \xi}(\xi_2)|  \\
&=&y_{\xi_2}(\xi_1)+o(|y_{\xi_2}(\xi_1)|) 
\end{eqnarray*}
as $\xi_1,\xi_2 \to \xi$ (where $\tilde \xi$ is ``between" $\xi$ and $\xi_2$).

Hereafter we might pass to subsequences without further notice. Let us split $\{1,\dots,N\}$ as $Z_0\cup \dots \cup Z_\ell$, where 
$$Z_0=\{j: \ |\xi_j^n-p_i|\geq c \hbox{ for all }i\},\quad Z_i=\{j: \ \xi_j^n\to p_i\} \quad  i=1,\dots,\ell.$$ 
We begin with the following two Lemmas.
\begin{lemma}\label{step1}
The following holds:
\begin{enumerate} \item[a)] if $\#Z_i=1$ for some $i=1,\dots,\ell$, then $\beta_1^n-\beta_2^n\to 0$;
\item[b)]  if $d_g(\xi_j^n,\xi_k^n)=o(1)$ for some $j,k\in Z_0$, $j\neq k$, then $\beta_1^n+\beta_2^n\to 0;$
\item[c)]  there exists $i\in\{1,\dots,\ell\}$ such that $\#Z_i\geq 2. $\end{enumerate}
\end{lemma}
\begin{proof} 
If $Z_{i}=\{j_0\}$, the identity \eqref{leone} with $j=j_0$ in the coordinate system $y_{p_{i}}$ gives
$$(\beta_2^n-\beta_1^n)\alpha_{i}\Gamma_{j_0}\frac{y_{p_{i}}(\xi_{j_0}^n)}{|y_{p_{i}}(\xi_{j_0}^n)|^2}=O(1),$$
and then $\rm a)$ follows. Next, let $j_0, k_0\in Z_0$, $j_0\neq k_0$, be such that  $d_g(\xi_{j_0}^n,\xi_{k_0}^n)=o(1)$. We may assume
$$d_g(\xi_{j_0}^n,\xi_{k_0}^n)=\min_{j,k\in Z_0 \atop j\neq k}d_g(\xi_j^n,\xi_k^n)\qquad \forall \ n\in\N.$$
Letting
$$I=\big\{j\in Z_0: \ d_g(\xi_{j}^n,\xi_{j_0}^n)\sim d_g(\xi_{j_0}^n,\xi_{k_0}^n)\big\} \cup \{j_0 \},$$
where $\sim$ denotes sequences of same order as $n\to +\infty$, observe that by construction 
$$d_g(\xi_j^n,\xi_k^n) \sim d_g(\xi_{j_0}^n,\xi_{k_0}^n) \qquad \forall \ j,k\in I,\,j\neq k$$ and
$d_g(\xi_{j_0}^n,\xi_{k_0}^n)=o(d_g(\xi_j^n,\xi_k^n))$ for all $j\in I$ and $k\in Z_0 \setminus I$, by which $$\nabla_{\xi_j}G(\xi_j^n, \xi_k^n)=o\bigg(\frac{1}{d_g(\xi_{j_0}^n,\xi_{k_0}^n)}\bigg)\qquad \forall \ j\in I,\, k\in Z_0 \setminus I.$$
The identities \eqref{leone} read in the coordinate system $y_{\xi_{j_0}^n}$ as
\beq\label{qqw100} 2(\beta_1^n+\beta_2^n)\Gamma_j\sum_{k\in I\atop k\neq j}\Gamma_k\frac{y_{\xi_{j_0}^n}(\xi_{j}^n)-y_{\xi_{j_0}^n}(\xi_k^n)}{|y_{\xi_{j_0}^n}(\xi_j^n)-y_{\xi_{j_0}^n}(\xi_k^n)|^2}=o\bigg(\frac{1}{d_g(\xi_{j_0}^n,\xi_{k_0}^n)}\bigg)\qquad \forall \ j\in I\eeq 
in view of \eqref{1538}. Since
\beq\label{idii} 2 \sum_{j,k\in I\atop j\neq k}\Gamma_j\Gamma_k\frac{\langle z_j-z_k, z_j-z\rangle}{|z_j-z_k|^2}=2 \sum_{j,k\in I\atop j< k}\Gamma_j\Gamma_k=\sum_{j,k\in I\atop j \neq k}\Gamma_j\Gamma_k
\eeq
for all $z_j,z \in \R^2$, by taking the inner product of \eqref{qqw100} with $y_{\xi_{j_0}^n}(\xi_j^n)$ and summing up in $j\in I$ we deduce that 
$$(\beta_1^n+\beta_2^n)\sum_{j,k\in I\atop j \neq  k}\Gamma_j\Gamma_k
=o(1) .$$ 
Since  $j_0,\, k_0\in I$, we get $\beta_1^n+\beta_2^n=o(1)$, and $\rm b)$ follows. Finally, if  $\# Z_{i} \leq 1$ were true for all $i=1,\dots,\ell$, by \eqref{ops} we would get that  $\# Z_{i_0}=1$ for some $i_0=\{1,\dots,\ell\}$. On the other hand, thanks to \eqref{ops} we also have that 
$d_g(\xi_j^n,\xi_k^n)=o(1)$ for some $j,k\in Z_0$, $j\neq k$. Then, by $\rm a)$ and $\rm b)$ we would derive $\beta_1^n+\beta_2^n=o(1)$, $\beta_1^n-\beta_2^n=o(1)$, in contradiction with \eqref{ops}.
\end{proof}
\begin{lemma} \label{step2}
If $\#Z_i\geq 2$ for some $i=1,\dots,\ell$, then $d_g(\xi^n_{j},p_{i})=O(d_g(\xi^n_{j},\xi_{k}^n))$ for all  $j,k\in Z_i$, $j\neq k$. 
\end{lemma}
\begin{proof} 
By contradiction, assume the existence of $j_0, k_0 \in Z_i$, $j_0 \neq k_0$, such that \beq\label{mindi} \frac{d_g(\xi^n_{j_0},\xi_{k_0}^n)}{d_g(\xi^n_{j_0},p_{i})}=\min_{j,k \in Z_i \atop j \neq k }\frac{d_g(\xi^n_{j},\xi_{k}^n)}{d_g(\xi^n_{j},p_{i})}\to 0.\eeq
Letting
$$I=\big\{j\in Z_i: \ d_g(\xi_{j}^n,\xi_{j_0}^n)\sim d_g(\xi_{j_0}^n,\xi_{k_0}^n)\big\}\cup \{j_0 \},$$
observe that for $j \in I$ by construction 
$$d_g(\xi_j^n, p_{i})\sim d_g(\xi_{j_0}^n,p_{i}) $$
and 
\begin{equation} \label{0908}
d_g(\xi_j^n, \xi_k^n)\sim d_g(\xi_{j_0}^n,\xi_{k_0}^n) \quad \forall \ k\in I,\, k\neq j, \qquad d_g(\xi_{j_0}^n,\xi_{k_0}^n)=o(d_g(\xi_j^n,\xi_k^n)) \quad \forall \ k\in Z_i \setminus I,
\end{equation}
by which 
$$\nabla_{\xi_j}G(\xi_j^n, \xi_k^n)=o\bigg(\frac{1}{d_g(\xi_{j_0}^n,\xi_{k_0}^n)}\bigg)\qquad \forall \ j\in I,\ k\in Z_i \setminus I.$$
The identities \eqref{leone} in the coordinate system $y_{\xi_{j_0}^n}$ read as
\beq\label{mash}\begin{aligned}(\beta_2^n-\beta_1^n)&\alpha_i\Gamma_j\frac{y_{\xi_{j_0}^n}(\xi_j^n)-y_{\xi_{j_0}^n}(p_{i})}{|y_{\xi_{j_0}^n}(\xi_j^n)-y_{\xi_{j_0}^n}(p_{i})|^2}+2(\beta_1^n+\beta_2^n)\Gamma_j\sum_{k\in I\atop k\neq j}\Gamma_k\frac{y_{\xi_{j_0}^n}(\xi_j^n)-y_{\xi_{j_0}^n}(\xi_k^n)}{|y_{\xi_{j_0}^n}(\xi_j^n)-y_{\xi_{j_0}^n}(\xi_k^n)|^2}\\ &=o\bigg(\frac{\beta_1^n+\beta_2^n}{d_g(\xi_{j_0}^n,\xi_{k_0}^n)}\bigg)+o\Big(\frac{\beta_2^n-\beta_1^n}{d_g(\xi_{j_0}^n,p_{i})}\Big)+O(1)\qquad \forall \ j\in I \end{aligned}\eeq
in view of \eqref{1538}. By taking the inner product of \eqref{mash} with $y_{\xi_{j_0}^n}(\xi_j^n)$ and summing up in $j\in I$ we deduce that
$$(\beta_1^n +\beta_2^n)\sum_{j,k\in I\atop j \neq k}\Gamma_j\Gamma_k=
o(\beta_1^n+\beta_2^n)+(\beta_2^n-\beta_1^n)O\Big(\frac{d_g(\xi_{j_0}^n,\xi_{k_0}^n)}{d_g(\xi_{j_0}^n,p_{i})}\Big)+O(d_g(\xi_{j_0}^n,\xi_{k_0}^n))$$ 
thanks to \eqref{idii}, by which, using \eqref{mindi},  we get \beq\label{beta12}\beta_1^n+\beta_2^n=O\Big(\frac{d_g(\xi_{j_0}^n,\xi_{k_0}^n)}{d_g(\xi_{j_0}^n,p_{i})}\Big) \to 0.\eeq 
 By taking the inner product of \eqref{mash} with $y_{\xi_{j_0}^n}(\xi_j^n)-y_{\xi_{j_0}^n}(p_{i})$ and summing up in $j\in I$ we obtain that
$$(\beta_1^n-\beta_2^n)\alpha_i\sum_{j\in I}\Gamma_j=
(\beta_1^n +\beta_2^n)\sum_{j,k\in I\atop j \neq  k}\Gamma_j\Gamma_k+
o(1)+(\beta_1^n+\beta_2^n)o\Big(\frac{d_g(\xi_{j_0}^n,p_{i})}{d_g(\xi_{j_0}^n,\xi_{k_0}^n)}\Big)$$ 
thanks to \eqref{idii}. By \eqref{beta12} we arrive at $\beta_1^n\pm\beta_2^n \to 0$, in contradiction with \eqref{ops}.
\end{proof}

\medskip  If $\#Z_i \geq 2$, let us split $Z_i$ as $Y_1\cup \dots \cup Y_l$, $l\geq 1$, in such a way that for all $j\in Y_r$
\beq\label{mario2}
d_g(\xi_j^n,p_i)\sim d_g(\xi_k^n,p_i)\quad \forall \  k\in Y_r,\qquad 
d_g(\xi_j^n,p_i)=o(d_g(\xi_k^n,p_i))\quad \forall \  k\in Y_{r+1}\cup\dots \cup Y_l.\eeq 
Notice that by construction $d_g(\xi_j^n,\xi_k^n)\sim d_g(\xi_k^n,p_i)$ for all $j\in Y_r$, $k\in Y_{r+1}\cup\dots \cup Y_l$, and by Lemma \ref{step2} $d_g(\xi_j^n,\xi_k^n)\sim  d_g(\xi_k^n,p_i)$ for all $j, k\in Y_r$ ($j\neq k$), yielding 
\beq\label{mario3}d_g(\xi_j^n,\xi_k^n)\sim d_g(\xi_k^n,p_i)\qquad \forall \ j\in Y_r, \ k\in Y_{r}\cup\dots \cup Y_{l}, \;j\neq k .\eeq
Combining \eqref{mario2}-\eqref{mario3} we get
$$\nabla_{\xi_j}G(\xi_j^n, \xi_k^n)=o\bigg(\frac{1}{d_g(\xi_j^n,p_i)}\bigg)\qquad \forall \ j\in Y_r, \ k\in Y_{r+1}\cup \dots \cup Y_{l},$$
which inserted in \eqref{leone} (written in the coordinate system $y_{p_i}$) gives  that 
\beq\label{qw2}2(\beta_1^n+\beta_2^n)\Gamma_j\!\!\sum_{k\in Y_1\cup \dots \cup Y_r\atop k\neq j}\!\!\Gamma_k\frac{y_{p_i}(\xi_j^n)-y_{p_i}(\xi_k^n)}{|y_{p_i}(\xi_j^n)-y_{p_i}(\xi_k^n)|^2}=(\beta_1^n-\beta_2^n)\alpha_i\Gamma_j\frac{y_{p_i}(\xi_j^n)}{|y_{p_i}(\xi_j^n)|^{2}}+o\bigg(\frac{1}{d_g(\xi_j^n,p_i)}\bigg)\eeq 
for all $j\in Y_r$ in view of \eqref{1538}. Since  $|y_{p_i}(\xi_j^n)|=o(|y_{p_i}(\xi_j^n)-y_{p_i}(\xi_k^n)|) $ for all $ j\in Y_1\cup \dots \cup Y_{r-1}$ and $k\in Y_r$ owing to  \eqref{mario2}-\eqref{mario3}, we can compute 
 \beq\label{obs}\frac{\langle y_{p_i}(\xi_j^n)-y_{p_i}(\xi_k^n),y_{p_i}(\xi_j^n)\rangle }{|y_{p_i}(\xi_j^n)-y_{p_i}(\xi_k^n)|^2}\!=1+\frac{\langle y_{p_i}(\xi_j^n)-y_{p_i}(\xi_k^n),y_{p_i}(\xi_k^n)\rangle }{|y_{p_i}(\xi_j^n)-y_{p_i}(\xi_k^n)|^2}\!=1+o(1)\eeq 
for all $j\!\in\! Y_r$ and $k\!\in\! Y_1\cup\dots \cup Y_{r-1}$. By taking the inner product of \eqref{qw2} with $y_{p_i}(\xi_j^n)$ and summing up in $j\in Y_r$ we get that
\beq\label{gra2}(\beta_1^n+\beta_2^n)\bigg(
\sum_{j,k\in Y_r\atop j\neq k}\Gamma_j\Gamma_k
+2 \sum_{j\in Y_r\atop k\in Y_1\cup \dots \cup Y_{r-1}}\Gamma_j\Gamma_k
\bigg)=(\beta_1^n-\beta_2^n)\alpha_i\sum_{j\in Y_r}\Gamma_j+o(1)\qquad \forall \ r\geq 1\eeq
thanks to \eqref{idii} and \eqref{obs}. 
Since $\# Z_{i}\geq 2$, notice that the coefficient in brackets on the left hand side of  \eqref{gra2} is positive when  $r=l$, and then $\beta_1^n-\beta_2^n$ and $\beta_1^n+\beta_2^n$ are positively proportional up to higher order terms. By \eqref{ops} and \eqref{gra2} (with $r=l$) we deduce that
\beq\label{bell0}\beta_1^n-\beta_2^n,\, \beta_1^n+\beta_2^n, |\beta_1^n|,\, |\beta_2^n|\geq c>0,\eeq 
taking into account that  $\beta_2^n =o(1)$ would imply $\beta_1^n=1+o(1)$ and  consequently, by \eqref{gra2} (with  $r=1$),
$$\sum_{j,k\in Y_1\atop j\neq k}\Gamma_j\Gamma_k=\alpha_i\sum_{j\in Y_1}\Gamma_j,$$  
 contradicting \eqref{compact}. Setting
\beq\label{gra3}a=\lim_{n \to +\infty} \frac{\beta_2^n}{\beta_1^n} \neq 0,\eeq
 let us evaluate the different pieces of the energy as follows: 
$$\begin{aligned}
&\sum_{j,k\in Z_i \atop j\neq k} \Gamma_j\Gamma_kG(\xi_j^n,\xi_k^n)-\alpha_i\sum_{j\in Z_i}\Gamma_jG(\xi_j^n,p_i)+O(1)\\
&=\sum_{r=1}^l\sum_{j,k\in Y_r \atop j\neq k} \Gamma_j\Gamma_kG(\xi_j^n,\xi_k^n)+2\sum_{r=1}^l\!\!\sum_{j\in Y_r\atop k\in Y_1\cup \dots \cup Y_{r-1}}\!\!\!\!\!\!\Gamma_j\Gamma_kG(\xi_j^n,\xi_k^n)-\alpha_i\sum_{r=1}^l \sum_{j\in Y_r}\Gamma_j G(\xi_j^n,p_i)+O(1)\\
&=-\frac{1}{2\pi} \sum_{r=1}^l \!\bigg[\sum_{j,k\in Y_r\atop j\neq k} \Gamma_j\Gamma_k\log d_g(\xi_j^n,\xi_k^n) +2 \!\!\!\!\sum_{j\in Y_r\atop k\in Y_1\cup\dots\cup Y_{r-1}}\!\!\!\!\!\!\Gamma_j\Gamma_k \log d_g(\xi_j^n,\xi_k^n)-\alpha_i \sum_{j\in Y_r}\Gamma_j\log d_g(\xi_j^n,p_i)\bigg].
\end{aligned}$$
Since $d_g(\xi_j^n,\xi_k^n)\sim d_g(\xi_j^n,p_i)$ for all $j\in Y_r$ and $k\in Y_1\cup \dots \cup Y_{r}$ in view of \eqref{mario3}, by \eqref{gra2} and \eqref{gra3} we have that
\begin{eqnarray*}
&&\sum_{j,k\in Y_r\atop j\neq k} \Gamma_j\Gamma_k\log d_g(\xi_j^n,\xi_k^n)+2\sum_{j\in Y_r\atop k\in Y_1\cup\dots\cup Y_{r-1}}\Gamma_j\Gamma_k\log d_g(\xi_j^n,\xi_k^n)-\alpha_i\sum_{j\in Y_r}\Gamma_j\log d_g(\xi_j^n,p_i)\\ 
&&=  \bigg(\sum_{j,k\in Y_r\atop j\neq k} \Gamma_j\Gamma_k+2\sum_{j\in Y_r\atop k\in Y_1\cup\dots\cup Y_{r-1}} \Gamma_j\Gamma_k-\alpha_i\sum_{j\in Y_r}\Gamma_j\bigg)\log d_g(\xi_{j_r}^n,p_i)+O(1)\\ 
&&= -(a+o(1))\Big(\sum_{j,k\in Y_r\atop j\neq k} \Gamma_j\Gamma_k+2\sum_{j\in Y_r\atop k\in Y_1\cup\dots\cup Y_{r-1}}\Gamma_j\Gamma_k+\alpha_i\sum_{j\in Y_r}\Gamma_j\Big)\log d_g(\xi_{j_r}^n,p_i) +O(1),
\end{eqnarray*}
where $j_r\in Y_r$ is fixed. We have thus proved that 
\beq\label{puff}\frac{1}{a}\bigg(\sum_{j,k\in Z_i \atop j\neq k} \Gamma_j\Gamma_k G(\xi_j^n,\xi_k^n)-\alpha_i\sum_{j\in Z_i}\Gamma_jG(\xi_j^n,p_i)\bigg)\to -\infty\eeq 
for all $Z_i $ with $\#Z_i\geq 2$. 
By \eqref{bell0} we deduce $\#Z_i \neq 1$ for all $i=1,\dots,\ell$ and 
$G(\xi_j^n,\xi_k^n)=O(1)$ for all $(j,k)\notin  \bigcup_{i=1}^\ell (Z_i \times Z_i)$
according to Lemma \ref{step1}-$\rm a)$ and $\rm b)$, then we conclude that
\begin{eqnarray*}
\frac{1}{a} \cal H(\bbm[\xi]_n)&=&\frac{1}{a}\bigg[\sum_{j,k=1\atop j\neq k}^N \Gamma_j\Gamma_kG(\xi_j^n, \xi_k^n)-\sum_{i=1}^\ell\alpha_i\sum_{j=1}^N\Gamma_jG(\xi_j^n,p_i)\bigg]+O(1) \\ 
&=& \frac{1}{a} \sum_{i=1}^\ell \bigg[\sum_{j,k\in Z_i \atop j\neq k} \Gamma_j \Gamma_k G(\xi_j^n, \xi_k^n)- \alpha_i\sum_{j\in Z_i}\Gamma_j G(\xi_j^n,p_i)\bigg]+O(1) \to -\infty
\end{eqnarray*}
by \eqref{puff}, in contradiction with \eqref{boupsi}.

\appendix
\renewcommand{\theequation}{\Alph{section}.\arabic{equation}}
\section{Proof of Theorem \ref{maximizingN}}
\setcounter{equation}{0}  
Setting $a_i=1+[\alpha_i]^-$, the  aim of this section is to compute the maximum
\beq\label{max}N:=\max\{ N_1+\dots+N_\ell:  N_i \in \N\cup\{0\},\ N_i+N_{i+1}\leq a_{i+1} \: \forall \ i=1,\dots,\ell \},\eeq
with the convention $a_{\ell+1}=a_1$ and $N_{\ell+1}=N_1$. The cases $\ell=2,3$ are easier to handle and will be treated later in details. From now on, let us assume $\ell\geq 4$.
Notice that for any $i=1,\dots, \ell$  
$$0\leq N_i \leq \min\{a_i,a_{i+1}\}$$  and 
 for any $1=1,\dots,\ell-1$ $$N_{i+1}\leq \min\{a_{i+1}-N_i,a_{i+2}-N_{i+2} \}.$$
Therefore, setting $J_i=[0, \min\{a_i,a_{i+1}\}]\cap (\N\cup\{0\})$, we have that \eqref{max}  
 can be rewritten as
$$N\!=\max_{N_i \in J_i, \ i \hbox{ \tiny odd}}\Big( N_1+\min\{a_2-N_1,a_3-N_3\}+N_3+\dots+\min \{a_\ell-N_{\ell-1},a_1-N_1\}\Big)$$
when $\ell$ is even and
\beq\label{oddN}N\!=\max_{N_i\in J_i,\ i\hbox{ \tiny odd}\atop N_1+N_\ell\leq a_1} \Big( N_1+\min\{a_2-N_1,a_3-N_3\}+N_3+\dots+\min \{a_{\ell-1}-N_{\ell-2},a_\ell-N_\ell\}+N_\ell \Big)\eeq
when $\ell$ is odd.

For the sake of clarity, we fix the following three Lemmas.

\begin{lemma}\label{1150} Let $f(t)=\min\{\alpha,\beta-t\}+t+\min\{\gamma-t,\delta\}$ for $0\leq t \leq T$. Then $$
\max_{[0,T]} f=\min \{\alpha+\gamma,\beta+\gamma,\alpha+\delta+T,\beta +\delta\}.
$$

\end{lemma}
\begin{proof}
For $t \in \mathbb{R}$ we can write
$$f(t)=\left \{ \begin{array}{ll} \alpha+\delta+t& \hbox{if } t \leq \min\{\beta-\alpha,\gamma-\delta \}\\
\min\{\beta+\delta,\alpha+\gamma\} & \hbox{if }\min\{\beta-\alpha,\gamma-\delta \}\leq t  \leq \max\{\beta-\alpha,\gamma-\delta \}\\
\beta+\gamma-t &\hbox{if }t  \geq \max\{\beta-\alpha,\gamma-\delta \}, \end{array} \right. $$
yielding 
\begin{eqnarray*} \max_{[0,T]}f&=& \left\{ \begin{array}{ll}  \min \{\alpha+\delta+T,\beta+\delta,\alpha+\gamma\} &\hbox{if }\min\{\beta-\alpha,\gamma-\delta\}\geq 0\\
\min\{\beta+\delta,\alpha+\gamma\} & \hbox{if }\min\{\beta-\alpha,\gamma-\delta \}\leq 0  \leq \max\{\beta-\alpha,\gamma-\delta \}\\
\beta+\gamma & \hbox{if } \max\{\beta-\alpha,\gamma-\delta \}\leq 0 \end{array} \right.\\
&=&\min \{\alpha+\gamma,\beta+\gamma,\alpha+\delta+T,\beta +\delta\}
\end{eqnarray*}
as claimed.
\end{proof}
\medskip  Let us fix $2\leq  k\leq\frac\ell2$ and  consider the  numbers $c_k, d_k, f_k, g_k$ defined in the introduction. 
We get \beq\label{uno}c_2=a_2+a_4,\ d_2=a_3+a_4,\ f_2=a_2+\min\{a_3,a_4\}+a_5,\ g_2=a_3+a_5.\eeq
\begin{lemma} The following identities hold:
\begin{equation} \label{2245}
N=\max_{N_1 \in J_1}\Big( \min \{ c_{\frac{\ell}{2}},d_{\frac{\ell}{2}}+N_1,f_{\frac{\ell}{2}}- N_1,g_{\frac{\ell}{2}} \} \Big)
\end{equation}
when $\ell$ is even and
\begin{equation}\label{2246}
N=\max_{N_1 \in J_1, N_\ell \in J_\ell\atop N_1+N_\ell \leq a_1}  \Big( \min \{ c_{\frac{\ell-1}{2}}+N_\ell,d_{\frac{\ell-1}{2}}+N_1+N_\ell,f_{\frac{\ell-1}{2}},g_{\frac{\ell-1}{2}}+N_1\} \Big) 
\end{equation}
when $\ell$ is odd. 
\end{lemma}
\begin{proof}
We claim that for every $2\leq k\leq \frac\ell2$ we have 
\begin{equation} \label{1229}\begin{aligned}
& \max_{N_i\in J_i\,  i=3,\dots ,2k-1\,\hbox{\tiny odd}}\Big(\! \min\{a_2-N_1,a_3-N_3\}+N_3+\dots +\min \{a_{2k}-N_{2k-1},a_{2k+1}-N_{2k+1}\}\!\Big)  \\
&= \min \{ c_k-N_1,d_k,f_k-N_1- N_{2k+1},g_k-N_{2k+1} \}. \end{aligned}
\end{equation}
Indeed, \eqref{1229} is valid for $k=2$ owing to \eqref{uno}, and the validity of \eqref{1229} with index $k$ implies
$$\begin{aligned}
& \max_{N_i\in J_i\,  i=3,\dots ,2k+1\,\hbox{\tiny odd}} \Big( \min\{a_2-N_1,a_3-N_3\}+N_3+\dots+\min \{a_{2k+2}-N_{2k+1},a_{2k+3}-N_{2k+3} \} \Big)\\
&=\max_{N_{2k+1} \in J_{2k+1} }\Big(  \min \{ \alpha_k,\beta_k-N_{2k+1}\}+N_{2k+1}+\min \{a_{2k+2}-N_{2k+1},a_{2k+3}-N_{2k+3}\}\Big),
\end{aligned}$$
where $\alpha_k=\min \{ c_k-N_1,d_k\}$ and $\beta_k=\min \{ f_k-N_1,g_k\}$. By Lemma \ref{1150} we have that
$$\begin{aligned}
&\max_{N_{2k+1} \in J_{2k+1}}\Big( \min\{ \alpha_k,\beta_k-N_{2k+1}\}+N_{2k+1}+\min\{a_{2k+2}-N_{2k+1},a_{2k+3}-N_{2k+3}\}\Big)\\
&=\min \{ \alpha_k+a_{2k+2},  \beta_k+a_{2k+2}, \alpha_k+\min\{a_{2k+1},a_{2k+2}\}+a_{2k+3}-N_{2k+3}, \beta_k+a_{2k+3}-N_{2k+3} \}.
\end{aligned}$$
The validity of \eqref{1229} with index $k+1$ is now achieved through the identities
\begin{eqnarray*}
c_{k+1}&=&\min\{c_k+a_{2k+2},f_k+a_{2k+2}\},\qquad  d_{k+1}=\min\{d_k+a_{2k+2},g_k+a_{2k+2}\},\\
f_{k+1}&=&\min\{c_k+\min\{a_{2k+1},a_{2k+2}\}+a_{2k+3},f_k+a_{2k+3}\},\\
g_{k+1}&=& \min\{d_k+\min\{a_{2k+1},a_{2k+2}\}+a_{2k+3},g_k+a_{2k+3}\},
\end{eqnarray*}
which follow by direct inspection of the definition of numbers $c_k,d_k, f_k, g_k$. Finally, by \eqref{1229} we immediately get the thesis of the Lemma.
\end{proof}
\begin{lemma} \label{ineq} The following inequalities  hold:
\begin{itemize}
\item[\rm(a)] $c_k+g_k\leq d_k+f_k$ for all $2\leq k\leq \frac\ell2$;
\item[\rm(b)] $\min\{ c_{\frac{\ell}{2}},g_{\frac{\ell}{2}} \}\leq d_{\frac{\ell}{2}}+\min \{a_1,a_2\}$ when $\ell$ is even.
\end{itemize}
\end{lemma}
\begin{proof} The inequality in $(a)$ does hold for $k=2$ thanks to \eqref{uno} and its validity at step $k$ implies that
$$\begin{aligned}
c_{k+1}+g_{k+1}&=\min\{c_k+a_{2k+2},f_k+a_{2k+2}\}+
\min\{d_k+\min\{a_{2k+1},a_{2k+2}\}+a_{2k+3},g_k+a_{2k+3}\}\\
&\leq \left\{ \begin{array}{l} c_k+a_{2k+2}+g_k+a_{2k+3}\leq (d_k+a_{2k+2})+(f_k+a_{2k+3})\\
(d_k+a_{2k+2})+ (c_k+\min\{a_{2k+1},a_{2k+2}\}+a_{2k+3})\\
(g_k+a_{2k+2})+(f_k+a_{2k+3})\\
(g_k+a_{2k+2})+(c_k+\min\{a_{2k+1},a_{2k+2}\}+a_{2k+3})\end{array} \right.,
\end{aligned}$$
yielding 
$$\begin{aligned}c_{k+1}+g_{k+1}&\leq  \min\{d_k+a_{2k+2},g_k+a_{2k+2}\}+
\min\{c_k+\min\{a_{2k+1},a_{2k+2}\}+a_{2k+3},f_k+a_{2k+3}\}\\ &=d_{k+1}+f_{k+1}.\end{aligned}$$
By induction the inequality in (a) is true for all $k \geq 2$, which implies the validity of
\begin{eqnarray*}
&&\min\Big\{c_{\frac{\ell}{2}}-d_{\frac{\ell}{2}},g_{\frac{\ell}{2}}-d_{\frac{\ell}{2}}, \frac{1}{2}\Big(f_{\frac{\ell}{2}}-d_{\frac{\ell}{2}}\Big)   \Big\}=\min\{c_{\frac{\ell}{2}},g_{\frac{\ell}{2}}\}-d_{\frac{\ell}{2}} \\
&&
\max\Big\{f_{\frac{\ell}{2}}-\min\{c_{\frac{\ell}{2}},g_{\frac{\ell}{2}}\},\frac{1}{2}\Big(f_{\frac{\ell}{2}}-d_{\frac{\ell}{2}}\Big)  \Big\}=f_{\frac{\ell}{2}}-\min\{c_{\frac{\ell}{2}},g_{\frac{\ell}{2}}\}
\end{eqnarray*}
for $\ell$ even, in view of $2\min \{c_{\frac{\ell}{2}},g_{\frac{\ell}{2}}\} \leq f_{\frac{\ell}{2}}+d_{\frac{\ell}{2}}$. Concerning $(b)$ notice that
\begin{equation} 
 s_k(J)=(1-\chi_J(2))\chi_J(1) \min\{a_3,a_4\}+p_k(J^\#)\label{2326},\eeq
\beq s_k(J)=a_{2k} \chi_J(k)+(1-\chi_J(k))[a_{2k+1}+\chi_J(k-1) \min\{a_{2k-1},a_{2k}\}]+q_k(J_\#)\label{2335}
\end{equation}
for some functions $p_k,q_k$ and $J^\#=J \cap \{2,\dots,k\}$, $J_\#=J \cap \{1,\dots,k-1 \}$. Given $J \subset \{1,\dots,\frac{\ell}{2} \}$ so that $1\notin J$, $\frac{\ell}{2} \in J$, for $\hat J=\{1\} \cup J$ by \eqref{2326} we have that
\begin{eqnarray*}
c_{\frac{\ell}{2}} &\leq& a_2+s_{\frac{\ell}{2}}(\hat J)=(a_2-a_3+(1-\chi_{\hat J}(2)) \min \{a_3,a_4\})+\big(a_3+p_{\frac{\ell}{2}}({\hat J}^\#)\big) \\
&\leq& (a_2-a_3+\min\{a_3,a_4\})+\big(a_3+p_{\frac{\ell}{2}}(J^\#)\big)
\leq a_2+(a_3+s_{\frac{\ell}{2}}(J))
\end{eqnarray*}
yielding  $c_{\frac{\ell}{2}}\leq a_2+d_{\frac{\ell}{2}}$. Similarly, for $\bar J=J \setminus \{\frac{\ell}{2} \}$ by \eqref{2335} we get that
\begin{eqnarray*}
g_{\frac{\ell}{2}} &\leq& a_3+s_{\frac{\ell}{2}}(\bar J)=\Big(a_1+\chi_{\bar J}\Big(\frac{\ell}{2}-1\Big) \min \{a_{\ell-1},a_\ell\}\Big)+\big(a_3+q_{\frac{\ell}{2}}(\bar J_\#)\big) \\
&\leq& \Big(a_1+\chi_{J}\Big(\frac{\ell}{2}-1\Big) \min \{a_{\ell-1},a_\ell\}-a_\ell\Big)+\big(a_3+s_{\frac{\ell}{2}}(J)\big)
\leq a_1+\big(a_3+s_{\frac{\ell}{2}}(J)\big)
\end{eqnarray*}
providing  $g_{\frac{\ell}{2}}\leq a_1+d_{\frac{\ell}{2}}$. In conclusion, we have shown that
$$\min\{ c_{\frac{\ell}{2}},g_{\frac{\ell}{2}} \}\leq d_{\frac{\ell}{2}}+\min \{a_1,a_2\}$$ 
and $(b)$ is thus established. 

\end{proof}

\noindent{\bf{Proof of Theorem \ref{maximizingN}}}. Thanks to inequalities (a)-(b) of Lemma \ref{ineq}, for $\ell$ even we have that
$$\min \{ c_{\frac{\ell}{2}},d_{\frac{\ell}{2}}+N_1,f_{\frac{\ell}{2}}- N_1,g_{\frac{\ell}{2}}\}
=\left\{ \begin{array}{ll} d_{\frac{\ell}{2}}+N_1& \hbox{if }N_1\leq \min\{c_{\frac{\ell}{2}},g_{\frac{\ell}{2}} \}-d_{\frac{\ell}{2}}\\
\min\{c_{\frac{\ell}{2}},g_{\frac{\ell}{2}} \}& \hbox{if }\min\{c_{\frac{\ell}{2}}, g_{\frac{\ell}{2}} \}-d_{\frac{\ell}{2}} \leq N_1 \leq f_{\frac{\ell}{2}}-\min\{c_{\frac{\ell}{2}},g_{\frac{\ell}{2}}  \}\\
f_{\frac{\ell}{2}}-N_1& \hbox{if } N_1 \geq f_{\frac{\ell}{2}}-\min\{c_{\frac{\ell}{2}},g_{\frac{\ell}{2}}  \}
\end{array} \right.$$
with $\min\{c_{\frac{\ell}{2}}, g_{\frac{\ell}{2}} \}-d_{\frac{\ell}{2}} \leq \min\{a_1,a_2\}$, 
yielding 
\begin{equation} \label{2346}
N=\min \{ c_{\frac{\ell}{2}},g_{\frac{\ell}{2}}\}
\end{equation}
when $\ell$ is even in view of \eqref{2245}. Unfortunately, when $\ell$ is odd the expression of $N$ in \eqref{2246} can be simply reduced to
\begin{equation} \label{0002}
N=\max_{\hat a_1 -\min\{a_1,a_\ell\}\leq  N_1 \leq  \min\{a_1,a_2\}} \Big( \min \{ c_{\frac{\ell-1}{2}}+\hat a_1-N_1,d_{\frac{\ell-1}{2}}+\hat a_1,f_{\frac{\ell-1}{2}},g_{\frac{\ell-1}{2}}+N_1 \} \Big)
\end{equation}
because $N_1+N_\ell \leq \hat a_1:=\min \{a_1,\min\{a_1,a_2\}+\min\{a_1,a_\ell\} \}$ with the equality achieved for all the maximizers in \eqref{2246}.

\medskip  An interesting situation corresponds to the case where the $a_i$'s are ordered in an increasing way. To distinguish it from the general case, we will denote it by $b_1, \dots, b_\ell$. Given $J \subset \{1,\dots,k\}$ 
and $3\leq j \leq k$, $s_k(J)$ depends on $j-1$ only through the term
$$b_{2j-2}\chi_J(j-1) +(1-\chi_J(j-1))[b_{2j-1}+\chi_J(j-2) b_{2j-3}] +(1-\chi_J(j)) \chi_J(j-1) b_{2j-1} $$ 
which is minimized by the choice $j-1 \in J$ or $j-1 \notin J$ depending on whether $j \in J$ or not, respectively. The same holds if $j=2$. Therefore, the minimization in the definition of $c_k$ and $d_k$ is achieved by sets $J$ with $\{2,\dots,k\} \subset J$, yielding 
\begin{eqnarray} \label{1920}
c_k=\sum_{j=1}^k b_{2j},\ d_k=b_3+\sum_{j=1}^k b_{2j},
\end{eqnarray} 
whereas for $f_k$ and $g_k$ the minimizing sets $J$ satisfy $J \cap \{2,\dots,k\}=\emptyset$ and then
\begin{eqnarray} \label{1921}
f_k=b_2 +\sum_{j=1}^k b_{2j+1}, \: g_k=\sum_{j=1}^{k} b_{2j+1}.
\end{eqnarray} 
Since $g_{\frac{\ell}{2}}\leq c_{\frac{\ell}{2}}$ in view of \eqref{1920}-\eqref{1921}, for $\ell$ even \eqref{2346} becomes
\begin{equation} \label{2346bis}
N=g_{\frac{\ell}{2}}=\sum_{j=0}^{\frac{\ell-2}{2}} b_{2j+1}.
\end{equation}
Since $c_{\frac{\ell-1}{2}}\leq d_{\frac{\ell-1}{2}}$ and $g_{\frac{\ell-1}{2}}+b_1\leq f_{\frac{\ell-1}{2}}$ by \eqref{1920}-\eqref{1921}, \eqref{0002} gives that
\begin{eqnarray} 
N&=& \max\{ \min \{ c_{\frac{\ell-1}{2}}+b_1-N_1, g_{\frac{\ell-1}{2}}+N_1 \}: 0 \leq  N_1 \leq  b_1  \}\nonumber \\
&=& \min\Big\{ c_{\frac{\ell-1}{2}}+b_1,  \frac{1}{2}\Big(c_{\frac{\ell-1}{2}}+b_1+g_{\frac{\ell-1}{2}}\Big)\Big\} = \min\Big\{b_1+\sum_{j=1}^{\frac{\ell-1}{2}} b_{2j}, \frac{1}{2} \sum_{j=1}^l b_j\Big\}  \label{0002bis}
\end{eqnarray}
when $\ell$ is odd, in view of 
$$\frac{1}{2}(c_{\frac{\ell-1}{2}}+b_1-g_{\frac{\ell-1}{2}})
\leq \frac{b_1}{2}.$$
\hfill\fbox{}

\bigskip

Finally let us discuss the case $\ell=2,3,4$.  

\medskip \noindent {\fbox{$\ell=2$}} We clearly have that $N=\min \{a_1,a_2\}$.

\medskip  \noindent {\fbox{$\ell=3$}} By \eqref{oddN} we deduce
$$\begin{aligned}N&=\max\{\min\{a_2+N_3, a_3+N_1\}: \ N_1+N_3\leq a_1, 0\leq N_1\leq \min\{a_1,a_2\},\, N_3\leq \min\{a_1,a_3\}\}
\\ &=\max\{\min\{a_2+\hat a_1-N_1, a_3+N_1\}: \ \hat a _1-\min\{a_1,a_3\} \leq N_1\leq \min\{a_1,a_2\}\}
\end{aligned}$$
since  $N_1+N_3 \leq \hat a_1:=\min \{a_1,\min\{a_1,a_2\}+\min\{a_1,a_3\} \}$ with the equality achieved for all the maximizers.
Then, we compute 
$$N=\left\{\begin{array}{llll}
\frac{a_2+a_3+\hat a_1}{2}&\hbox{if}& \hat a_1-\min\{a_1,a_3\} \leq \frac{a_2-a_3+\hat a_1}{2}\leq \min\{a_1,a_2\}\\  a_3+\min\{a_1,a_2\}&\hbox{if} &\frac{a_2-a_3+\hat a_1}{2}\geq \min\{a_1,a_2\}\\ 
a_2+\min\{a_1,a_3\}&\hbox{if} &\frac{a_2-a_3+\hat a_1}{2}\leq \hat a _1-\min\{a_1,a_3\} .
\end{array}\right.$$
By discussing all the six possibilities for $(a_1,a_2,a_3)$ ($a_1<a_2<a_3$, $a_2<a_3<a_1$ and so on), we immediately realize that 
$N=\min\{b_1+b_2,\frac{b_1+b_2+b_3}{2}\}$, which actually corresponds to  \eqref{0002bis} for $\ell=3$ for the increasing ordering $b_1,b_2,b_3$. We  have thus proved that the maximal $N$ in \eqref{max} is independent of the order of $a_i$'s. 
 
\medskip  Instead of coupling the $a_i$'s in a consecutive way (ordered or not), with regards  to   Theorem \ref{thgammauno}  let us now couple the pair $a_1$, $a_2$ with $a_3$ and $a_3$ with $a_1$: in order to satisfy \eqref{192}, we decompose $N$ as $N=N_1+N_2+N_3$ with $$N_1+N_3\leq a_1,\;\;N_1+N_2+N_3\leq a_3,\;\;N_2\leq a_2.$$ We can easily see that the maximal $N$ satisfies
$N=\min\{a_1+a_2,a_3\}$, giving $N \leq \min\{b_1+b_2,\frac{b_1+b_2+b_3}{2}\}$. Since such a case represents the general situation for a non-consecutive coupling, we can summarize our discussion by saying that the consecutive increasing coupling $b_1,b_2,b_3$ gives rise to the best maximal $N=\min\{b_1+b_2,\frac{b_1+b_2+b_3}{2}\}$ among all the possible couplings (consecutive or not, increasing or not). Such a property is peculiar for $\ell=3$, as we will see below by discussing the case $\ell=4$.

\medskip  \noindent {\fbox{$\ell=4$}} By \eqref{2346} we have that
$N=\min \{ a_1+a_3,a_2+a_4\}$, and the optimal choice is $(a_1,a_2,a_3,a_4)=(b_2,b_1,b_3,b_4)$ which gives rise to $N=\min \{ b_2+b_3,b_1+b_4\}$. Since in general $\min \{ b_2+b_3,b_1+b_4\}>b_1+b_3$ (see \eqref{2346bis}), we see that the increasing ordering is no longer the optimal   among all the consecutive ones. Moreover, referring to  non-consecutive couplings in  Theorem \ref{thgammauno}, let us couple the pair $b_1$, $b_2$ with $b_4$, $b_3$ with $b_2$ and $b_4$ with $b_3$:  in order to satisfy \eqref{192} we need to require that $N=N_1+N_2+N_3+N_4$ with 
$$N_1\leq b_1,\: N_2+N_3 \leq b_2,\: N_3+N_4\leq b_3, \: N_1+N_2+N_4\leq b_4.$$ 
The particular choice $N_1=b_1$, $N_2=\min\{b_2, b_4-b_1\}$, $N_3=0$  and $N_4=\min\{b_3, b_4-b_1-\min\{b_2, b_4-b_1\} \}$ leads to $N\geq \min\{b_1+b_2+b_3,b_4\}$. Since $\min \{ b_1+b_2+b_3,b_4\} >\min \{ b_2+b_3,b_1+b_4\}$ when $b_1+b_2+b_3\leq b_4$, we also see that consecutive couplings are not in general the optimal among all the possible ones.

\end{document}